\documentclass[11pt]{article}
\usepackage[top=1in, bottom=1in, left=1in, right=1in]{geometry}
\usepackage{tikz}

\usepackage{amsmath,amsfonts,amssymb,amscd,amsthm,mathrsfs}
\usepackage{extarrows,textcomp}
\usepackage{graphicx,subfigure,epstopdf}
\usepackage{diagbox,multirow}
\usepackage{algorithm,algorithmic}
\usepackage{caption}
\usepackage{bm}
\usepackage{float}
\usepackage{lineno}
\usepackage{xcolor}
\usepackage{makeidx,hyperref}

\newtheorem{theorem}{Theorem}[section]
\newtheorem{lemma}{Lemma}[section]

\newtheorem{example}{Example}[section]
\newtheorem{remark}{Remark}{}
\numberwithin{equation}{section}

\DeclareMathOperator*{\argmin}{arg\,min}
\newcommand{\dx}{\,{\rm d}x}
\newcommand{\dd}{\,{\rm d}}

\makeatletter
\newcommand*{\extendadd}{
  \mathbin{
    \mathpalette\extend@add{}
  }
}
\newcommand*{\extend@add}[2]{
  \ooalign{
    $\m@th#1\leftrightarrow$%
    \vphantom{$\m@th#1\updownarrow$}
    \cr
    \hfil$\m@th#1\updownarrow$\hfil
  }
}
\makeatother

\begin{document}

\title{Int-Deep: A Deep Learning Initialized Iterative Method for Nonlinear Problems}

\author{Jianguo Huang and Haoqin Wang
\vspace{0.1in}\\
School of Mathematical Sciences, and MOE-LSC,\\
 Shanghai Jiao Tong University, Shanghai 200240, China.
  \vspace{0.1in}\\
Haizhao Yang
  \vspace{0.1in}\\
  Department of Mathematics, Purdue University\footnote{Current institute.}, West Lafayette, IN 47907,  USA\\
  Department of Mathematics, National University of Singapore\footnote{Part of the work was done in Singapore.}, Singapore
}

\maketitle

\begin{abstract}
This paper proposes a deep-learning-initialized iterative method (Int-Deep) for low-dimensional nonlinear partial differential equations (PDEs). The corresponding framework consists of two phases. In the first phase, an expectation minimization problem formulated from a given nonlinear PDE is approximately resolved with mesh-free deep neural networks to parametrize the solution space. In the second phase, a solution ansatz of the finite element method to solve the given PDE is obtained from the approximate solution in the first phase, and the ansatz can serve as a good initial guess such that Newton's method or other iterative methods for solving the nonlinear PDE are able to converge to the ground truth solution with high-accuracy quickly. Systematic theoretical analysis is provided to justify the Int-Deep framework for several classes of problems. Numerical results show that the Int-Deep outperforms existing purely deep learning-based methods or traditional iterative methods (e.g., Newton's method and the Picard iteration method).
\end{abstract}



\section{Introduction}\label{sec: introduction}

This paper is concerned with the efficient numerical method for solving nonlinear partial differential equations (PDEs) including a class of eigenvalue problems as special cases, which is a ubiquitous and important topic in science and engineering \cite{PhysRev1,nature,PhysRev3,PhysRev2,BASDEVANT198623,Miura,HYMAN1986113,10.2307/2004575}. As far as we know, there have developed many traditional and typical numerical methods in this area, e.g., the finite difference method, the spectral method, and the finite element method \cite{QuarteroniValli1994}. The first two methods are generally used for solving problems over regular domains while the latter one is particularly suitable for solving problems over irregular domains \cite{Bartels,ZienkiewiczTaylorZhu2005}. To achieve the numerical solution with the desired accuracy, one is often required to numerically solve the discrete problem formulated as a large-scale nonlinear system of nonlinear equations, which is time-consuming. In this case, one of the most critical issues is to choose the feasible initial guess so that the numerical solver (e.g. Newton's method) is convergent. On the other hand, for reducing the computational cost,  two grid methods are thereby devised in \cite{Xu1994,Xu1996}, which only require to solve a small-sized nonlinear system arising from the finite element discretization based on a coarse triangulation. However, the difficulty is that we even can not ensure if the nonlinear system has a solution in theory if the mesh size of the coarse triangulation is large enough.

Recently, science and engineering have undergone a revolution driven by the success of deep learning techniques that originated in computer science. This revolution also includes broad applications of deep learning in computational and applied mathematics. Many new branches in scientific computing have emerged based on deep learning in the past five years including new methods for solving nonlinear PDEs. There are mainly two kinds of deep learning approaches for solving non-linear PDEs:  mesh-based \cite{Mingui2003,Tang2017,Yuehaw2017,Yuehaw20182,Yuwei2018,Fan2019,Fan2019BCRNetAN} and mesh-free \cite{BERG2018,Carleo2017,Han8505,Khoo2018,RUDD2015,SIRIGNANO2018,RaissiPerdikarisKarniadakis2019}. In the mesh-based methods, deep neural networks (DNNs) are constructed to approximate the solution operator of a PDE, e.g., seeking a DNN that approximates the map mapping the coefficients (or initial/boundary conditions) of a PDE to the corresponding solution. After construction, the DNN can be applied to solve a specific class of PDEs efficiently. In the mesh-free methods, which probably date back to 1990's (e.g., see \cite{doi:10.1002/cnm.1640100303,712178}), DNNs are applied as the parametrization of the solution space of a PDE; then the solution of the PDE is identified via seeking a DNN that fits the constraints of the PDE in the least-squares sense or minimizes a variational problem formulated from PDEs. The key to the success of these approaches is the universal approximation capacity of DNNs \cite{kurkova1992,barron1993,yarotsky2017,yarotsky2018,ShenYangZhang20192} even without the curse of dimensionality for a large class of functions \cite{barron1993,montanelli2019a,montanelli2019b,Yang20192,LuShenYangZhang2020}.

Though the deep learning approach has made it possible to solve high-dimensional problems, which is a significant breakthrough in scientific computing,  to the best of our knowledge, the advantage of deep learning approaches over traditional methods in the low-dimensional region is still not clear yet. The main concern is the computational efficiency of these frameworks: the number of iterations in deep learning methods is usually large or the accuracy is very limited (e.g., typically $10^{-2}$ to $10^{-4}$ relative error). In order to overcome this difficulty, one is tempted to set up a more efficient neural network architecture (e.g., incorporating physical information in the structure designing \cite{doi:10.1063/1.4961454,doi:10.1063/1.5054310,Zhang:2018:ESP:3327345.3327356,Yuwei2018,GuYangZhou2020}, designing a more advanced learning algorithm for deep learning training \cite{MEADE199419,haomin}, or using a solution ansatz according to prior knowledge \cite{712178,SHEKARIBEIDOKHTI2009898,5061501}). However, the overall performance of these frameworks for nonlinear problems without any prior knowledge may still not be very convincingly efficient.

This paper proposes the Int-Deep framework from a new point of view for designing highly efficient solvers of low-dimensional nonlinear PDEs with a finite element accuracy leveraging both the advantages of traditional algorithms and deep learning approaches.  The Int-Deep framework consists of two phases { {as shown in Figure \ref{fig:flow}}}. In the first phase, an approximate solution to the given nonlinear PDE is obtained via deep learning approaches using DNNs of size $O(1)$ and $O(100)$ iterations, where $O(\cdot)$ means that the prefactor is independent of the final target accuracy in the Int-Deep framework, i.e., the accuracy of finite element methods. In particular, based on variational principles, we propose new methods to formulate the problem of solving nonlinear PDEs into an unconstrained minimization problem of an expectation over a function space parametrized via DNNs, which can be solved efficiently via batch stochastic gradient descent (SGD) methods due to the special form of expectation. Unlike previous methods in which the form of expectation is only derived for nonlinear PDEs related to variational equations, our proposed method can also handle those related to variational inequalities, providing a unified variational framework for a wider range of nonlinear problems.

\begin{figure}[H]
		\centering
		\includegraphics[width = 13 cm]{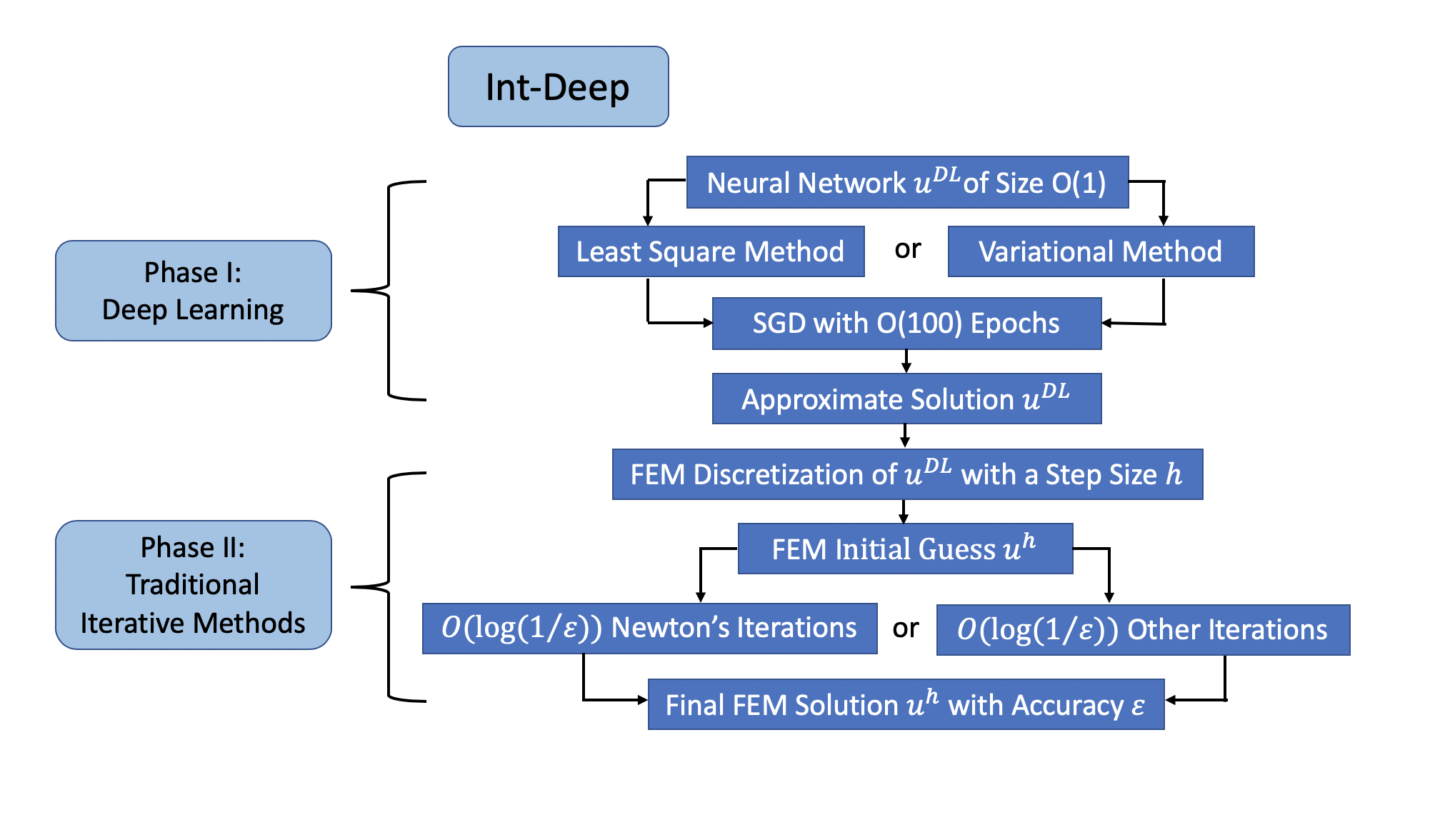}
		\caption{Computational flow of Int-Deep: Deep learning solvers in Phase I only require $O(1)$ computational cost since the network size is $O(1)$ and the number of epochs of the stochastic gradient descent (SGD) is $O(100)$. Empirically, traditional iterative methods converges to a solution with $O(\epsilon)$ accuracy in $O(\log(\frac{1}{\epsilon}))$ iterations in Phase II.}
		\label{fig:flow}
\end{figure}

In the second phase, the approximate solution provided by deep learning approaches can serve as a good initial guess such that traditional iterative methods (e.g., Newton's method for solving nonlinear PDEs and the shifted power method for eigenvalue problems) converge quickly to the ground truth solution with high-accuracy. The hybrid algorithm substantially reduces the learning cost of deep learning approach while keeping the quality of initial guesses for traditional iterative methods; good initial guesses enable traditional iterative methods to converge in $O(\log(\frac{1}{\epsilon}))$ iterations to the $\epsilon$ precision of finite element methods. In each iteration of the traditional iterative method, the nonlinear problem has been linearized and hence traditional fast solvers for linear systems can be applied depending on the underlying structure of the linear system. Therefore, as we shall see in the numerical section, the Int-Deep framework outperforms existing purely deep learning-based methods or traditional iterative methods, e.g., Newton's method and Picard iteration. Furthermore, systematic theoretical analysis is provided to characterize the conditions under which the Int-Deep framework converges, serving as a trial to change current deep learning research from trial-and-error to a suite of methods informed by a principled design methodology.

This paper will be organized as follows. In Section \ref{sec:dnn}, we briefly review the definitions of DNNs. In Section \ref{sec:I}, we introduce the expectation minimization framework for deep learning-based PDE solvers in the first phase of the Int-Deep framework. In Section \ref{sec:II}, as the second phase of Int-Deep, traditional iterative methods armed with good initial guesses provided by deep learning approaches will be introduced together with its theoretical convergence analysis, for clarity of presentation the proofs of which are left in Appendix. In Section \ref{sec: experiments}, a set of numerical examples will be provided to demonstrate the efficiency of the proposed framework and to justify our theoretical analysis. Finally in Section \ref{sec:con}, we summarize our paper with a short discussion.

\section{Deep Neural Networks (DNNs)}
\label{sec:dnn}

Mathematically, DNNs are a form of highly non-linear function parametrization via function compositions using simple non-linear functions \cite{IanYoshuaAaron2016}.  The validity of such an approximation method can be ensured by the universal approximation theorems of DNNs in \cite{kurkova1992,barron1993,yarotsky2017,yarotsky2018,ShenYangZhang2019,ShenYangZhang20192}. Let us introduce two basic neural network structures commonly used for solving PDEs below.

The first one is the so-called fully connected feed-forward neural network (FNN), which is a function in the form of a composition of $L$ simple nonlinear functions as follows:
\[
	\phi(\bm{x};\bm{\theta}):=\bm{a}^T \bm{h}_L \circ \bm{h}_{L-1} \circ \cdots \circ \bm{h}_{1}(\bm{x}),
\]
 where $\bm{h}_{\ell}(\bm{x})=\sigma\left(\bm{W}_\ell \bm{x} + \bm{b}_\ell \right)$ with $\bm{W}_\ell \in \mathbb{R}^{N_{\ell}\times N_{\ell-1}}$, $\bm{b}_\ell \in \mathbb{R}^{N_\ell}$ for $\ell=1,\dots,L$, $\bm{a}\in \mathbb{R}^{N_L}$, $\sigma$ is a non-linear activation function, e.g., a rectified linear unit (ReLU) $\sigma(x)=\max\{x,0\}$ or hyperbolic tangent function $\tanh(x)$. Each $\bm{h}_\ell$ is referred as a hidden layer,  $N_\ell$ is the width of the $\ell$-th layer, and $L$ is called the depth of the FNN. In the above formulation, $\bm{\theta}:=\{\bm{a},\,\bm{W}_\ell,\,\bm{b}_\ell:1\leq \ell\leq L\}$ denotes the set of all parameters in $\phi$, which uniquely determines the underlying neural network.


{  The second one is the so-called residual neural network (ResNet) introduced by He, Zhang, Ren and Sun in \cite{HeZhangRenSun2016}. We apply its variant defined recursively as follows:}
 \begin{eqnarray*}
 \bm{h}_0&=&\bm{V}\bm{x},\\
 \bm{g}_\ell&=&\sigma(\bm{W}_\ell\bm{h}_{\ell-1}+\bm{b}_{\ell}),
\qquad\ell=1,2,\dots,L,\\
 \bm{h}_\ell&=&\bm{\bar{U}}_\ell \bm{h}_{\ell-2}+\bm{U}_\ell\bm{g}_\ell, \quad\ell=1,2,\dots,L,\\
\phi(\bm{x};
 \bm{\theta})&=&\bm{a}^T\bm{h}_L,
 \end{eqnarray*}
 where $\bm{V}\in \mathbb{R}^{N_0\times d}$, $\bm{W}_\ell\in \mathbb{R}^{N_{\ell}\times N_{0}}$, $\bm{\bar{U}}_\ell\in \mathbb{R}^{N_{0}\times N_{0}}$, $\bm{U}_\ell\in \mathbb{R}^{N_{0}\times N_{\ell}}$, $\bm{b}_\ell\in \mathbb{R}^{N_\ell}$ for $\ell=1,\dots,L$, $\bm{a}\in \mathbb{R}^{N_0}$, $\bm{h}_{-1}=\bm{0}$. Throughout this paper, we consider $N_0=N_\ell=N$ and $\bm{U}_\ell$ is set as the identity matrix in the numerical implementation of ResNets for the purpose of simplicity. Furthermore, as used in \cite{EYU2018}, we set $\bm{\bar{U}}_\ell $ as the identify matrix when $\ell$ is even and set $\bm{\bar{U}}_\ell =\bm{0}$ when $\ell$ is odd.

\section{Phase I of Int-Deep: Variational Formulas for Deep Learning}\label{sec:I}

In this section, we will present existing and our new analysis for reformulating nonlinear PDEs including eigenvalue problems to the minimization of expectation that can be solved by SGD. The analysis works for PDEs related to variational equations and variational inequalities, the latter of which is of special interest since there might be no available literature discussing this case to the best of our knowledge. Hence, our analysis could serve as a good reference for a wide range of problems.

{  Throughout this paper, we will use standard symbols and notations for Sobolev spaces and their norms/seminorms; we refer the reader to the reference \cite{Adams1975} for details. Moreover, the standard $L^2(D)$-inner product for a bounded domain $D$ is denoted by $(\cdot, \cdot)_{D}$. If the domain $D$ is the solution domain $\Omega$, we will drop out the dependence of $\Omega$ in all Sobolev norms/seminorms and $L^2(\Omega)$-inner product when there is no confusion caused.}

\subsection{PDE Solvers Based on DNNs}\label{sec:sol}

The general idea of deep learning-based PDE solvers is to treat DNNs as an efficient parametrization of the solution space of a PDE and the solution of the PDE is identified via seeking a DNN that fits the constraints of the PDE in the least-squares sense or minimizes the variational minimization problem related to the PDE. Let us use the following example to illustrate the main idea:
\begin{equation}\label{eqn:PDE}
\left\{
\begin{aligned}
	\mathcal{D}(u)&=f\quad \text{in }\Omega, \\
       \mathcal{B}(u)&=g\quad \text{on }\partial\Omega,
\end{aligned}
\right.
\end{equation}
where $\mathcal{D}$ is a differential operator and $\mathcal{B}$ is a boundary operator.

In the least squares type methods (LSM), a DNN $\phi(\bm{x};\bm{\theta}^*)$ is constructed to approximate the solution $u(\bm{x})$ for $\bm{x}\in\Omega$ via minimizing the square loss
\begin{equation}\label{eqn:PDEl2}
\bm{\theta}^*=\argmin_{\bm{\theta}} \mathcal{L}(\bm{\theta}):=
\mathbb{E}_{\bm{x}\in\Omega} \left[\left|\mathcal{D}\phi(\bm{x};\bm{\theta})-f(\bm{x})\right|^2\right] + \gamma \mathbb{E}_{\bm{x}\in \partial\Omega} \left[\left|\mathcal{B}\phi(\bm{x};\bm{\theta})-g(\bm{x})\right|^2\right],
\end{equation}
with a positive parameter $\gamma$. 

In the variational type methods (VM), \eqref{eqn:PDE} is solved via a variational minimization
\begin{equation}\label{eq: nonlinear}
	u^*=\argmin_{u\in H} J(u),
\end{equation}
where the Hilbert space $H$ is an admissible space, and $J(u)$ is a nonlinear functional over $H$. Then, the solution space $H$ is parametrized via DNNs, i.e., $H\approx \{\phi(\bm{x}; \bm{\theta})\}_{\bm{\theta}}$, where $\phi$ is a DNN with a fixed depth $L$ and width $N$. After parametrization, \eqref{eq: nonlinear} is approximated by the following problem:
\begin{equation}\label{eq: optim theta}
	\bm{\theta}^*=\argmin_{\bm{\theta}} J(\phi(\bm{x}; \bm{\theta})).
\end{equation}
In general, $J(\phi(\bm{x}; \bm{\theta}))$ can be formulated as the sum of several integrals over several sets $\{\Omega_i\}_{i=1}^p$, each of which corresponds to one equation in \eqref{eqn:PDE}:
\begin{equation}\label{eq: integral}
	J(\phi(\bm{x}; \bm{\theta}))
	= \sum_{i=1}^p \int_{\Omega_i} F_i(\bm{x}; \bm{\theta})\dx
	= \sum_{i=1}^p |\Omega_i| \mathbb{E}_{\bm{\xi}_i}\Big[F_i(\bm{\xi}_i; \bm{\theta})\Big],
\end{equation}
where $F_i(\cdot; \bm{\theta)}$ is a function related to a variational constraint on $\phi(\bm{x}; \bm{\theta})$, $\bm{\xi}_i$ is a random vector produced by the uniform distribution over $\Omega_i$, and $|\Omega_i|$ denotes the measure of $\Omega_i$. In addition, $\bm{\xi}_i$ ($1\leq i \leq p$) are mutually independent. Based on \eqref{eq: integral}, \eqref{eq: optim theta} can be expressed as
\begin{equation}\label{eq: solvable optim theta}
	\bm{\theta}^{*} =\argmin_{\bm{\theta}} \sum_{i=1}^p |\Omega_i| \mathbb{E}_{\bm{\xi}_i} \Big[F_{i}(\bm{\xi}_i; \bm{\theta})\Big].
\end{equation}

Both the LSM in \eqref{eqn:PDEl2} and VM in \eqref{eq: optim theta} can be reformulated to the expectation minimization problem in \eqref{eq: solvable optim theta}, which can be solved by the stochastic gradient descent (SGD) method or its variants (e.g., Adam \cite{KB2014}). In this paper, we refer to \eqref{eq: solvable optim theta} as the expectation minimization framework for PDE solvers based on deep learning.

Although the convergence of SGD for minimizing the expectation in \eqref{eq: solvable optim theta} is still an active research topic, empirical success shows that SGD can provide a good approximate local minimizer of \eqref{eqn:PDEl2} and \eqref{eq: solvable optim theta}. This completes the algorithm of using deep learning to solve nonlinear PDEs with equality constraints.

We would like to emphasize that when the PDE in \eqref{eqn:PDE} is nonlinear, its solutions might be the saddle points of \eqref{eq: integral} making it very challenging to identify its solutions via minimizing \eqref{eq: integral}. Hence, we will use the LSM in \eqref{eqn:PDEl2} for PDEs associated with variational equations. It deserves to point out that \eqref{eqn:PDEl2} is also regarded as a variational formulation, referred to as the least-squares variational principle in \cite{BochevGunzburger2009} in contrast with the usual Ritz variational principle.

\subsection{Minimization Problems for Variational Inequalities}\label{sub:VI}

Variational inequalities are a class of important nonlinear problems, frequently encountered in various industrial and engineering applications \cite{DuvautLions1976,Glowinski1984}. Unlike previous methods in which the form of expectation is only derived for nonlinear PDEs with variational equations, we propose an expectation minimization framework also suitable for variational inequalities.

The abstract framework of an elliptic variational inequality of the second kind can be described as follows (cf. \cite{Glowinski1984}). Find $u\in H $ such that
	\begin{equation}
	\label{VI}
	a(u,v-u)+j(v)-j(u)\geq \langle f,v-u\rangle,\quad v \in H,
	\end{equation}
where $H$ is a Hilbert space equipped with the norm $\|\cdot\|_H$, and $\langle\cdot,\cdot\rangle$  stands for the duality pairing between $H'$ and $H$, with $H'$ being the dual space of $H$; $a(\cdot,\cdot)$ is a continuous, coercive and
symmetric bilinear form over $H$; $j(\cdot): H\rightarrow \mathbb{\overline{R}}=\mathbb{R}\cup \{\pm \infty\}$
is a proper, convex and lower semi-continuous functional.

As shown in \cite{Glowinski1984}, the above problem has a unique solution under the stated conditions on the problem data. Moreover, it can be reformulated as the following minimization problem:
\begin{equation}
\label{VI-minimization}	
u=\argmin_{v\in H} J(v)=\frac 1 2 a(v,v)-\langle f,v\rangle+j(v),
\end{equation}
which naturally falls into our expectation minimization framework in \eqref{eq: solvable optim theta}.

Let us discuss the simplified friction problem as an example, where the nonlinear PDE is given by
\begin{equation}
\label{eq: simplifiedfriction}
	\begin{cases}
		-\Delta u + u = f \quad & \mbox{in} \ \Omega,
		\\
		|\partial _{\bm{n}}u| \leq g, \ u \partial_{\bm{n}}u + g|u| = 0 \quad & \mbox{on} \ \Gamma_C,
		\\
		u = 0 \quad & \mbox{on} \  \Gamma_D,
	\end{cases}	
\end{equation}
where $\bm{n}$ is the unit outward normal to $\partial\Omega$; $\Gamma_C\subset \partial\Omega$ denotes the friction boundary, and $\Gamma_D = \partial \Omega \setminus \Gamma_C$; $f \in L^2(\Omega)$ and $g \in L^2(\Gamma_C)$ are two given functions.  \eqref{eq: simplifiedfriction} can be expressed as an elliptic variational inequality in the form \eqref{VI} or \eqref{VI-minimization} by choosing
\[
	H=V_D = \{ \phi \in H^1(\Omega): \phi = 0 \ \mbox{on} \ \Gamma_D \}
\]
and
\[
	a(\phi,\chi) = (\nabla \phi, \nabla \chi) + (\phi,\chi), \quad j(\phi) = (g,|\phi|)_{\Gamma_C},\quad \phi,\; \chi\in V_D.
\]

Hence, the minimization problem of the nonlinear PDE \eqref{eq: simplifiedfriction} is given by
\begin{equation}
\label{opt: simplifiedfriction}
	u = \argmin_{\phi \in V_D} J_1 (\phi),
\end{equation}
where
\begin{align*}
	J_1(\phi)
	& = \frac{1}{2}\| \phi \|_1^2 - (f, \phi) + j(\phi)
	\\
	& = |\Omega| \mathbb{E}_{\bm{\xi}_1}\Big[ \frac{1}{2} \big( |\nabla_{\bm{\xi}_1} \phi(\bm{\xi}_1;\bm{\theta})|^2 + \phi^2(\bm{\xi}_1;\bm{\theta}) \big) - f(\bm{\xi}_1)\phi(\bm{\xi}_1;\bm{\theta}) \Big]
	+ |\Gamma_C| \mathbb{E}_{\bm{\xi}_2} \Big[ g(\bm{\xi}_2)|\phi(\bm{\xi}_2; \bm{\theta})| \Big],
\end{align*}
{ 
where $\bm{\xi}_1$ and $\bm{\xi}_2$ are random vectors following the uniform distribution over $\Omega$ and $\Gamma_C$, respectively.
}


We can also use the penalty method to remove the { constraint condition} in the admissible space $V_D$, giving rise to an easier unconstrained minimization problem for deep learning. Then, the problem \eqref{opt: simplifiedfriction} is modified as
\begin{equation}\label{eq:VP2}
	u = \argmin_{\phi \in V_1} J_2 (\phi),
\end{equation}
where
\begin{align*}
	V_1 = H^1(\Omega)
	\quad
	J_2(\phi)=J_1(\phi)+\gamma \|\phi\|_{0,\Gamma_D}^2
	&= |\Omega| \mathbb{E}_{\bm{\xi}_1}\Big[ \frac{1}{2} \big( |\nabla_{\bm{\xi}_1} \phi(\bm{\xi}_1;\bm{\theta})|^2 + \phi^2(\bm{\xi}_1;\bm{\theta}) \big) - f(\bm{\xi}_1)\phi(\bm{\xi}_1;\bm{\theta}) \Big]
	\\
	& \quad + |\Gamma_C| \mathbb{E}_{\bm{\xi}_2} \Big[ g(\bm{\xi}_2)|\phi(\bm{\xi}_2; \bm{\theta})| \Big]
	+ |\Gamma_D| \mathbb{E}_{\bm{\xi}_3} \Big[ \gamma \phi^2(\bm{\xi}_3;\bm{\theta}) \Big],
\end{align*}
with $\gamma$ denoting a penalty parameter to be determined feasibly {  and $\bm{\xi}_3$ being a random vector following the uniform distribution over $\Gamma_D$}. The minimization problem in \eqref{eq:VP2} is of the form of expectation minimization in \eqref{eq: solvable optim theta}.


\subsection{Minimization Problems for Eigenvalue Problems}\label{sub:eig}
Now we discuss how to evaluate the smallest eigenvalue and its eigenfunction for a positive self-adjoint differential operator, e.g.,
\begin{equation}
\label{eq: eigenvalue Dirichlet}
	\begin{cases}
		- \nabla(p(\bm{x})\nabla u)+q(\bm{x})u = \lambda u \quad & \mbox{in} \  \Omega,
		\\
		u =0 \quad  & \mbox{on} \  \partial\Omega,
	\end{cases}
\end{equation}
where $p(\bm{x})\in C^1(\bar{\Omega})$ and there exist two positive constants $p_1\ge p_0$ such that $p_0\le p(\bm{x})\le p_1$ for all $\bm{x}\in \bar{\Omega}$, and $q(\bm{x})\in C(\bar{\Omega})$ is nonnegative over $\Omega$. Here and in what follows, $\bar{\Omega}$ denotes the closure of $\Omega$.

The solution $(\lambda,u)$ is governed by the following variational problem
\begin{equation}
\label{eigenfunction}
u = \argmin_{\phi\in V} J_3(\phi),
\end{equation}
where $V = H_0^1(\Omega)$ and
\[
	J_3(\phi) = \frac{a(\phi,\phi)}{\| \phi \|_0^2} + \gamma(\phi(\bm{x}_0) - 1 )^2;
	\quad
	a(\phi,\chi)=\int_{\Omega}[p(\bm{x})\nabla \phi\cdot \nabla\chi+q(\bm{x})\phi\ \chi] \dx,\quad \phi,\ \chi\in V;
\]
$\gamma$ is any given positive number, and $\bm{x}_0$ is any given point in the interior of $\Omega$.

If $u_*\in V$ is a solution, we have by the variational principle for eigenvalue problems  \cite{AmbrosettiMalchiodi2007} that
\begin{equation}
\label{eigenfunction1}
u_*=\argmin_{\phi\in V}\frac{a(\phi,\phi)}{\| \phi \|_0^2}.
\end{equation}
Let $\alpha$ be a constant such that $(\alpha u_*)(\bm{x}_0)=1$ and $u=\alpha u_*$. By \eqref{eigenfunction1}, $J_3(u)\le \frac{a(\phi,\phi)}{\| \phi \|_0^2}\le J_3(\phi)$ for all $\phi\in V$. The converse can be proved similarly. Furthermore, it is easy to check that
\[
  J_3(\phi)
	= \frac{a(\phi(\bm{x};\bm{\theta}),\phi(\bm{x};\bm{\theta}))}{\| \phi(\bm{x};\bm{\theta}) \|_0^2}+\gamma(\phi(\bm{x}_0;\bm{\theta})-1)^2
 {
	= \frac{\mathbb{E}_{\bm{\xi}} \Big[ p(\bm{\xi})|\nabla_{\bm{\xi}}\phi(\bm{\xi};\bm{\theta})|^2+q(\bm{\xi})\phi^2(\bm{\xi};\bm{\theta}) \Big] } {\mathbb{E}_{\bm{\eta}} \Big[ \phi^2(\bm{\eta};\bm{\theta}) \Big]}+\gamma(\phi(\bm{x}_0;\bm{\theta})-1)^2,
  }
\]
where $\bm{\xi}$ and $\bm{\eta}$ are i.i.d. random vectors following the uniform distribution over $\Omega$.

In fact, \eqref{eigenfunction} is a constrained minimization problem, that is
\begin{align*}
	& u = \argmin_{\phi\in V_1} J_3(\phi),
	\\
	\mbox{s.t.} \quad
	&u = 0   \quad  \mbox{on} \ \partial \Omega,
\end{align*}
where $V_1 = H^1(\Omega)$ as before. Here, we exploit the idea from Jens Berg and Kaj Nystr$\ddot{\text{o}}$m \cite{BERG2018} to reformulate the above problem as an unconstrained minimization one. To this end, we construct the neural network function as follows:
\[
	\phi(\bm{x};\bm{\theta}) = B(\bm{x})\psi(\bm{x};\bm{\theta}),
\]
where $B(\bm{x})$  is a known smooth function such that the boundary $\partial\Omega$ can be parametrized by $B(\bm{x})=0$, and  { $\psi(\bm{x};\bm{\theta})$} is any function in $V_1$. So \eqref{eigenfunction} becomes the expectation minimization:
\begin{equation}
\label{eq:VP3}
	u = \argmin_{\psi\in V_1} J_3(B(\bm{x})\psi(\bm{x};\bm{\theta})).
\end{equation}

Once we have obtained the eigenfunction $u(\bm{x})$, the corresponding eigenvalue is $\lambda=a(u,u)/\|u\|_0^2$. Note that the variational formulation developed here is different from the one devised by E and Yu \cite{EYU2018} and the expectation form in our formulation makes it easier to implement SGD.

The proposed method evaluates the smallest eigenvalue and its eigenfunction not only  for linear eigenvalue problems, but also for nonlinear eigenvalue problems. To generalize this idea for nonlinear eigenvalue problems, let us consider the nonlinear Schr$\ddot{\text{o}}$dinger equation called the Gross-Pitaevskii (GP) equation as an example:
\begin{equation}\label{eq: GP Dirichlet}
	\begin{cases}
		-\Delta u + V(\bm{x})u + \beta u^3= \lambda u &\quad \mbox{in} \ \Omega,
		\\
		u =0 &\quad \mbox{on} \ \partial\Omega, \\
		\| u \|_0 =1,
	\end{cases}
\end{equation}
where the real-valued potential function $V(\bm{x}) \in L^p(\Omega)$ for some real number $p>1$, and $\beta$ is a positive number.

According to~\cite{Cances2010}, the ground state non-negative solution $(\lambda, u)$ of \eqref{eq: GP Dirichlet} is unique and governed by the following variational problem
\begin{equation}\label{GP variational 1}
	u = \argmin \{ J_4(\phi): \phi \in H_0^1(\Omega),  \| \phi\|_0 =1 \},
\end{equation}
where
\[
	J_4(\phi) = (\nabla \phi, \nabla \phi)+ (V(\bm{x})\phi, \phi) + \frac{\beta}{2}(\phi^3,\phi).
\]

To relax the constraint $\| \phi \|_0 = 1$, we consider the minimization problem for $\phi/\|\phi\|_0$ with $\phi\in V$, and reformulate the variational problem \eqref{GP variational 1} in the form
\begin{equation} \label{GP variational 2}
	u = \argmin_{\phi \in H_0^1(\Omega)} J_5(\phi),
\end{equation}
where
\[
	J_5(\phi) = \frac{(\nabla \phi, \nabla \phi)+ (V(\bm{x})\phi, \phi)}{\| \phi \|_0^2} + \frac{\beta}{2}\frac{(\phi^3,\phi)}{\| \phi \|_0^4} + \gamma(\phi(\bm{x}_0) - 1 )^2;
\]
$\bm{x}_0$ is any given point in the interior of $\Omega$.
Moreover, the functional $J_5(\phi)$ can be written as the expectation framework as follow.
{ 
\begin{align*}
  J_5(\phi(\bm{x};\bm{\theta})) &=
	\frac{(\nabla_{\bm{x}} \phi(\bm{x};\bm{\theta}), \nabla_{\bm{x}} \phi(\bm{x};\bm{\theta}))+(V(\bm{x})\phi(\bm{x};\bm{\theta}), \phi(\bm{x};\bm{\theta}))}{\| \phi(\bm{x};\bm{\theta}) \|_0^2} + \frac{\beta}{2}\frac{(\phi^3(\bm{x};\bm{\theta}),\phi(\bm{x};\bm{\theta}))}{\| \phi(\bm{x};\bm{\theta}) \|_0^4} \\
	&\quad + \gamma(\phi(\bm{x}_0;\bm{\theta}) - 1 )^2 \\
	&=\frac{\mathbb{E_{\bm{\xi}}}\Big[ | \nabla_{\bm{\xi}} \phi(\bm{\xi};\bm{\theta})|^2 + V(\bm{\xi})\phi^2(\bm{\xi};\bm{\theta}) \Big]}{\mathbb{E_{\bm{\eta}}} \Big[ \phi^2(\bm{\eta};\bm{\theta}) \Big]} + \frac{\beta}{2|\Omega|}\frac{\mathbb{E_{\bm{\xi}}}\Big[\phi^4(\bm{\xi};\bm{\theta}) \Big]}{\mathbb{E_{\bm{\eta};\bm{\zeta}}} \Big[ \phi^2(\bm{\eta};\bm{\theta}) \phi^2(\bm{\zeta};\bm{\theta}) \Big]} + \gamma(\phi(\bm{x}_0;\bm{\theta}) - 1 )^2,
\end{align*}
}
where $\bm{\xi}, \bm{\eta}, \bm{\zeta}$ are i.i.d. random vectors produced by the uniform distribution over $\Omega$.

Similarly, we can also eliminate the boundary constraint following the ideas in linear eigenvalue problems. Then the variational problem \eqref{GP variational 1} is rewritten as
\begin{equation}\label{GP variational 3}
 {
	u = \argmin_{\psi \in V_1} J_5(B(\bm{x})\psi(\bm{x};\theta)),
}
\end{equation}
where $V_1 = H^1(\Omega)$ and $B(\bm{x})$ is the same as of linear eigenvalue problems.

It is noted that if $u$ is the solution of \eqref{GP variational 3}, the eigenfunction is its normalization: $u/\|u\|_0$ (still denoted by $u$ for simplicity).  Thus, the smallest eigenvalue can be computed by
\[
\lambda = (\nabla_{\bm{x}} u, \nabla_{\bm{x}} u)+(V(\bm{x})u, u) + \beta(u^3,u).
\]

\section{Phase II of Int-Deep: Traditional Iterative Methods}\label{sec:II}
We propose to use deep learning solutions as initial guesses so as to achieve a high-accurate solution by traditional iterative methods in a few iterations (e.g., Newton's method for solving nonlinear systems \cite{StoerBulirsch2002}\cite{GaoYangMeza2009} or the two grid methods in the context of finite elements \cite{Xu1994,Xu1996,XuZhou2001}\cite{CancesChakirHeMaday}). In fact, these ideas have led to high-performance methods for solving semilinear elliptic problems and eigenvalue problems with the effectiveness shown by mathematical theory and numerical simulation. The key analysis of this idea is to characterize the conditions under which deep learning solutions can help traditional iterative methods converge quickly. We will provide several classes of examples and the corresponding analysis to support this idea as follows.
{ 
In what follows, $u^{DL}$ denotes the numerical solution obtained by the deep learning algorithm, and $I_h$ is the usual nodal interpolation operator (cf. \cite{BrennerScott2008,Ciarlet1978}).}

\subsection{Semilinear Elliptic Equations with Equality Constrains}\label{sec:sl}

Consider the following semilinear elliptic equation
\begin{equation}\label{eq: semilinear}
	\begin{cases}
		-\Delta u + f(\bm{x},u) = 0 &\quad \mbox{in} \; \Omega,
		\\
		u = 0 &\quad \mbox{on} \; \partial\Omega,
	\end{cases}	
\end{equation}
where $\Omega$ is a bounded convex polygon in $\mathbb{R}^d$ ($d=1,2$), and $f(\bm{x},u)$ is a sufficiently smooth function. For simplicity, we use $f(u)$ for $f(\bm{x},u)$  and $f'$ for $f_u$ in what follows.

Let $V=H_0^1(\Omega)$. Then the variational formulation of problem \eqref{eq: semilinear} is to find $u\in V$ such that
\begin{equation}
\label{v-form}
a(u,\chi):=(\nabla u, \nabla \chi) + (f(u), \chi)=0,\quad \chi\in V.
\end{equation}

For any $v\in L^{\infty}(\Omega)$, define
\begin{equation}
\label{bilinear}
a_v(\phi,\chi) := (\nabla \phi, \nabla \chi) + (f'(v)\phi, \chi),\quad \phi, \  \chi\in V.
\end{equation}

As in \cite{Xu1994,Xu1996}, we assume that problem \eqref{eq: semilinear} (equivalently, problem \eqref{v-form})  satisfies the following conditions:
\begin{itemize}
	\item[A1]  Any solution $u$ of \eqref{eq: semilinear} has the regularity $u \in W^{2,\infty}(\Omega)$.

	\item[A2]  Let $u \in W^{2,\infty}(\Omega)$ be a solution of \eqref{eq: semilinear}. Then there exists a positive constant $C^u$ such that
	\[
		C^u \| \phi \|_1 \leq \sup_{\chi\in V} \frac{a_u(\phi,\chi)}{\| \chi \|_1},
		\quad \phi \in V.
	\]
\end{itemize}

We next consider the finite element method for solving problem \eqref{v-form}. Let $\mathcal{T}_h$ be a quasi-uniform and shape regular triangulation of $\Omega$ into $K$. Write $h_K = \text{diam}(K)$ and $h:= \max_{K\in\mathcal{T}_h} h_K$. Introduce the Courant element space by
\begin{equation}\label{eq:Vh}
	V_h = \{ \phi\in C(\bar{\Omega}): \phi|_K \in \mathbb{P}_1(K) \mbox{ for all } K\in \mathcal{T}_h  \} \cap V,
\end{equation}
where $\mathbb{P}_1(K)$ denotes the function space consisting of all linear polynomials over $K$. Then the finite element method is to find $u^h \in V_h$ such that
\begin{equation}
\label{fem-semilinear}
	a(u^h,\chi)=0, \quad \chi \in V_h.
\end{equation}

Now, let us introduce Int-Deep for solving the problem in \eqref{fem-semilinear}. Concretely speaking, we choose $I_h u^{DL}$ as the initial guess {  and obtain a highly accurate approximate solution by Newton's method.}

\begin{algorithm}[H]
	\caption{A hybrid Newton's method for semilinear problems}\label{alg: semilinear}
	\textbf{Input}: the target accuracy $\epsilon$, the maximum number of iterations $N_{\max}$,
	the approximate solution in a form of a DNN $u^{DL}$ in Phase I of Int-Deep.\\
	\textbf{Output}: $u^h = u^h_{k+1}$.
	\begin{algorithmic}
		\STATE {\textbf{Initialization}: Let $u_0^h = I_h u^{DL}$, $k=0$, and $e_k=1$;}
		\WHILE {$e_k > \epsilon$ and $k < N_{\max}$}
		\STATE { Find $v_k^h \in V_h$ such that
		 	\[
 				(\nabla v_k^h, \nabla \chi) + (f'(u_k^h)v_k^h, \chi) = -(\nabla u_k^h, \nabla \chi) - (f(u_k^h), \chi),
 				\quad  \chi \in V_h.
			\]
			 Let $u_{k+1}^h = u_{k}^h + v_k^h$.}
		\STATE {$e_{k+1} = \| u_{k+1}^h - u_k^h \|_0 / \| u_k^h \|_0$, $k = k+1$.}
		\ENDWHILE
	\end{algorithmic}
\end{algorithm}

Next, we turn to discuss the convergence of the  Algorithm~\ref{alg: semilinear}. In the theorem below, we introduce a discrete maximum norm $\|\cdot\|_{0,\infty,h}$ to quantify errors. Let $\Omega_h$ consist of all the vertices of the triangulation $\mathcal{T}_h$ of $\Omega$. Then for any $v\in C(\bar{\Omega})$, $\|v\|_{0,\infty,h}=\max_{\bm x \in \Omega_h}|v(\bm x)|$. 

{ 
In what follows, to simplify the presentation, for any two quantities $a$ and $b$, we write ``$a\lesssim b$" for ``$a\le C b$", where $C$ is a generic positive constant independent of $h$, which may take different values at different occurrences. Moreover, any symbol $C$ or $c$ (with or without superscript or subscript) denote positive generic constants independent of the finite element mesh size $h$.
}

Let us define a neighborhood of $u$ as follows:
\begin{equation}
\label{neighborhood}
B(u) = \{v \in V: \|v-u\|_{0,\infty} \leq C_1^u \},
\end{equation}
where $C^u_1$ is a positive constant given in Lemma \ref{lem: infsup}. Then we have the following theorem showing the convergence of Algorithm \ref{alg: semilinear}.

\begin{theorem}
\label{thm: semilinear PDE}
Let $u \in W^{2,\infty}(\Omega)$ be a solution of \eqref{eq: semilinear} and $u_k^h$  the function sequence formed by  Algorithm~\ref{alg: semilinear}. Let $B(u)$ be a neighborhood of $u$ given by \eqref{neighborhood}. Assume $u_k^h \in B(u)$ for $k=0,1,\cdots$. Write $\delta =\|u - u^{DL}\|_{0,\infty,h}$. Then there exist two positive constants {  $\bar{h}_0$} ($\bar{h}_0<h_2$) and {  $\bar{\delta_0}$} such that if $h<\bar{h}_0$ and $\delta<\bar{\delta_0}$, there holds
\begin{align*}
	& \| u - u_k^h \|_{0,\infty} \lesssim h^2 + \beta_1^{2^k} \quad \quad \quad \quad \quad \quad \,\,\,  {\rm for} \ d =1,
	\\
	& \| u - u_k^h \|_{0,\infty} \lesssim h^2|\log(h)| + h^{-2/p}\beta_2^{2^k} \quad {\rm for} \ d =2,
 \end{align*}
where $h_2$ is a positive constant given in Lemma \ref{lem: error estimates},
\[
\beta_1 = c_0c_1(h^2 + \delta)<1,\quad \beta_2 = c_2c_3(h^2|\log(h) + \delta|)<1,
\]
with $c_i$ ($0\le i\le 3$) as positive constants.
\end{theorem}

The proof of Theorem \ref{thm: semilinear PDE} can be found in Appendix A. According to the second-order convergence in $\beta_1$ and $\beta_2$ in the above theorem, one can use only a few Newton's iterations to achieve a numerical solution with the same accuracy as in the finite element solution $u^h$. The forthcoming numerical results will demonstrate this theoretical estimate. {  As shown in reference \cite{Xu1994} (see also Lemma \ref{convergence} in Appendix A), we find that under the condition $h<\bar{h}_0<h_2$ given in the above theorem, the finite element method \eqref{fem-semilinear} has a unique solution $u^h$ around a given exact solution $u$.} Moreover, the analysis developed here can be applied to high order finite element methods.

\subsection{Eigenvalue Problems}\label{subsec: Eigenvalue}

In this subsection, we propose and analyze Int-Deep for solving eigenvalue problems in the similar spirit of the two grid method due to Xu and Zhou \cite{XuZhou2001}. We first discuss how to design efficient methods for linear eigenvalue problems. For simplicity, let us consider the following problem
\begin{equation}\label{eq: eigen Dirichlet}
	\begin{cases}
		-\Delta u = \lambda u &\quad \mbox{in} \ \Omega,
		\\
		u =0 &\quad \mbox{on} \ \partial\Omega.
	\end{cases}
\end{equation}

The variational problem for \eqref{eq: eigen Dirichlet} is given as follows. Find the smallest number $\lambda$ and a nonzero function $u\in V$ such that
 \begin{equation}
 \label{v-form1}
		a(u,\chi) = \lambda (u, \chi), \quad \chi\in V,
\end{equation}
where $V=H^1_0(\Omega)$ and $a(u,\chi)=(\nabla u,\nabla \chi)$. Without loss of generality, assume $\|u\|_0=1$.

In the discretization, we use the same finite element space $V_h$ as given in the last subsection. Hence, the finite element method for \eqref{v-form1} is to find the smallest number $\lambda^h$ and a nonzero function $u^h\in V_h$ such that
 \begin{equation}
 \label{fem-eigenvalue}
		a(u^h,\chi) = \lambda^h (u^h, \chi), \quad \chi\in V_h.
\end{equation}

Now, let us introduce Int-Deep for solving the problem in \eqref{fem-eigenvalue}. Concretely speaking, we choose the normalized $I_h u^{DL}$ as the initial guess {  and obtain a highly accurate approximate solution by a certain iterative scheme.}
The iterative scheme is motivated by the two grid method for eigenvalue problems due to Xu and Zhou \cite{XuZhou2001}. The method is essentially the power method for eigenvalue problems and the eigenvalue computation is accelerated by the Rayleigh quotient.

\begin{algorithm}
	\caption{Int-Deep for linear eigenvalue problems}\label{alg: eigen}
	\textbf{Input}: the target accuracy $\epsilon$, the maximum number of iterations $N_{\max}$, the approximate solution in a form of a DNN $u^{DL}$ in Phase I of Int-Deep.\\
	\textbf{Output}: $u^h=u^h_{k+1}, \lambda^h = \lambda^h_{k+1}$.
	\begin{algorithmic}
		\STATE {\textbf{Initialization}: Let $u_0^h = I_h u^{DL} / \| I_h u^{DL} \|_0$, $\lambda_0^h = |u_0^h|_1^2/ \| u_0^h \|_0^2$, $k=0$, $e_k=1$;}
		\WHILE {$e_k > \epsilon$ and $k < N_{\max}$}
		\STATE {Find $u_{k+1}^h \in V_h$ such that
		 	\begin{align*}
 				a(u_{k+1}^h,\chi) &=\lambda_{k}^h(u_k^h, \chi) , \quad \chi \in V_h,
 				\\
 				\lambda_{k+1}^h &= \frac{|u_{k+1}^h|_1^2}{\| u_{k+1}^h \|_0^2};
			\end{align*}
			}
		\STATE {$e_{k+1} = \| u_{k+1}^h - u_k^h \|_0 / \| u_k^h \|_0$, $k = k+1$.}
		\ENDWHILE
	\end{algorithmic}
\end{algorithm}

{  In each iteration step of the ``while" loop in Algorithm \ref{alg: eigen}, typical numerical methods would normalize the approximate eigenfunction to enforce $\|u_k^h\|_0=1$, which might help to speedup the convergence. However, we will only prove the convergence of Algorithm \ref{alg: eigen} without the normalization step as follows for simplicity.} We define the elliptic projection operator $P_h$ that projects any $u\in V$ to $P_h u\in V_h$ such that $a(P_hu, \chi) = a(u,\chi), \quad \chi \in V_h$. The analysis of this operator is well-known (cf. \cite{BrennerScott2008,Ciarlet1978}) but we quote an important result (cf. \cite{BabuskaOsborn1989,BabuskaOsborn1991}) below.

\begin{lemma}\label{lem: eigenvlue}
	Let $(\lambda, u)$ be an eigenpair of \eqref{v-form1}. For any $\phi \in H_0^1(\Omega)\setminus\{0\}$, there holds
	\[
		\frac{a(\phi,\phi)}{(\phi,\phi)} - \lambda = \frac{a(\phi - u, \phi - u)}{(\phi,\phi)} - \lambda\frac{(\phi - u, \phi - u)}{(\phi, \phi)}.
	\]
\end{lemma}

With the help of the above results, we can derive the following theorem.

\begin{theorem}
\label{thm: eigenvalue problem}
Let $\lambda$ be the smallest eigenvalue of \eqref{eq: eigen Dirichlet} and $u \in H^2(\Omega)$ be the corresponding eigenfunction with $\|u\|_0=1$, respectively. Denote $\tilde{\delta} =\max\{| \lambda - \lambda_0^h|, \|u-u^{DL}\|_{0,\infty,h}\} $. Assume that the sequence $\|u_k^h\|_0$ is bounded below and above by two positive constants $\varepsilon_0$ and $\varepsilon_1$, respectively.
{ 
Then there exist two positive constants $\tilde{h_0}$ and $\tilde{\delta_0}$ such that if $h < \tilde{h_0}$ and $\tilde{\delta} < \tilde{\delta_0}$, there holds
\begin{equation}
\label{eigen-estimate}
 |\lambda - \lambda_k^h| + \| u - u_{k}^h \|_0 \lesssim \beta_3^{2^k} + h^2,
\end{equation}
where $\beta_3 = c_4c_6(\tilde{\delta} + h^2)$, $c_4$ and $c_6$ are generic positive constants independent of the finite element mesh size $h$.}
\end{theorem}

{ 
Note that  $\|u_0^h\|_0=1$ and only very few iterations are required to derive the desired numerical solution in Algorithm \ref{alg: eigen} in practical implementation. Hence, it is reasonable to assume the sequence $\|u_k^h\|_0$ is bounded below and above by two positive constants. This assumption is also verified by our numerical experiments.
}

The proof of Theorem \ref{thm: eigenvalue problem} can be found in Appendix B.  Similar to Theorem \ref{thm: semilinear PDE}, in view of the second-order convergence in $\beta_3$ in the above theorem, one can use only a few iterations to achieve a numerical solution with the same accuracy as for the related finite element method. The forthcoming numerical results will demonstrate this theoretical estimate.


Now, let us turn to the nonlinear eigenvalue problem. For simplicity consider the Gross-Pitaevskii equation \eqref{eq: GP Dirichlet}.
The weak form of \eqref{eq: GP Dirichlet} is given as follow. Find the smallest number $\lambda$ and a non-negative function $u\in H_0^1(\Omega)$ such that
\begin{equation}\label{nonlinear eig variational2}
	\begin{cases}
	a(u,\chi) + (f(u), \chi) = \lambda(u, \chi), \quad \chi \in H_0^1(\Omega),
	\\
	\quad \| u \|_0 = 1,
	\end{cases}
\end{equation}
where
\[
	a(u,\chi) = (\nabla u, \nabla \chi)+ (V(\bm{x})u, \chi), \quad  f(u) = \beta u^3.
\]
Once we have obtained the eigenfunction $u(x)$, then the corresponding eigenvalue is $\lambda=a(u,u) + (f(u), u)$.

To discretize the above problem, we use the same finite element space $V_h$ in \eqref{eq:Vh}. Hence, the finite element method for \eqref{nonlinear eig variational2} is to find the smallest number $\lambda^h$ and a non-negative function $u^h\in V_h$ such that
 \begin{equation}
 \label{nonlinear fem-eigenvalue}
 	\begin{cases}
		a(u^h,\chi) + (f(u^h), \chi)= \lambda^h (u^h, \chi), \quad \chi\in V_h,
		\\
		\quad 
    { \| u^h \|_0 =1.}
	\end{cases}
\end{equation}
The error analysis of the finite element method for \eqref{nonlinear fem-eigenvalue}  can be found in \cite{Zhou2004,Cances2010}.

Finally, let us introduce an accelerated iteration method for solving \eqref{nonlinear fem-eigenvalue}. Though an iteration scheme can be derived by the two grid method for nonlinear elliptic eigenvalue problems due to Canc\`{e}s, Chakir, and He \cite{CancesChakirHeMaday}, we propose another iteration scheme motivated by Newton's method for nonlinear eigenvalue problems due to Gao, Yang and Meza \cite{GaoYangMeza2009}. Here, we only derive the algorithm for nonlinear eigenvalue problems and we will analyze the convergence of Algorithm~\ref{alg: nonlinear eig} in future work.

\begin{algorithm}[H]
	\caption{A hybrid Newton's method for nonlinear eigenvalue problems}\label{alg: nonlinear eigen}\label{alg: nonlinear eig}
	\textbf{Input}: the target accuracy $\epsilon$, the maximum number of iterations $N_{\max}$, the approximate solution in a form of a DNN $u^{DL}$ in Phase I of Int-Deep.\\
	\textbf{Output}: $u^h=u^h_{k+1}, \lambda^h = \lambda^h_{k+1} $.
	\begin{algorithmic}
		\STATE {\textbf{Initialization}: Let $u_0^h = I_h u^{DL}/ \| I_h u^{DL} \|_0 $, $\lambda_0^h = a(u_{0}^h, u_{0}^h) + (f(u_0^h) ,u_0^h)$, $k=0$, $e_k=1$;}
		\WHILE {$e_k > \epsilon$ and $k < N_{\max}$}
		\STATE {Find $(v_k^h, \mu_k^h) \in V_h \times \mathbb{R}$ such that
		 \begin{align*}	
 			a(v_k^h, \chi) + (f^{\prime}(u_{k}^h)v_k^h, \chi ) - \lambda_k^h(v_{k}^h, \chi) - \mu_{k}^h(u_k^h, \chi) &= -a(u_k^h,\chi) - (f(u_k^h), \chi) + \lambda_k^h(u_k^h, \chi),
			\ \chi \in V_h,
			\\
 			2(v_{k}^h, u_k^h) &= 1 - (u_k^h, u_k^h).
		\end{align*}
			Let $u_{k+1}^h = u_k^h + v_{k}^h, \lambda_{k+1}^h = \lambda_k^h + \mu_{k}^h$.
			}
		\STATE {$e_{k+1} = \| u_{k+1}^h - u_k^h \|_0 / \| u_k^h \|_0$, $k = k+1$.}
		\ENDWHILE
	\end{algorithmic}
\end{algorithm}

\section{Numerical Experiments}\label{sec: experiments}
This section consists of two parts. In the first part, we provide various numerical examples to illustrate the performance of the proposed deep learning-based methods. As we shall see, deep learning approaches are capable of providing approximate solutions to nonlinear problems in $O(100)$ iterations. However, continuing the iteration cannot further improve accuracy. In the second part, we will investigate the numerical performance of the Int-Deep framework in terms of network hyperparameters and the convergence analysis in the previous section. { {Though the hyper-parameters of Int-Deep to guarantee a good initial guess should be problem-dependent, as we shall see in various numerical examples in this section,}} DNNs of size $O(1)$ and trained with $O(100)$ iterations can provide good initial guesses enabling traditional iterative methods to converge in $O(\log(\frac{1}{\epsilon}))$ iterations to the $\epsilon$ precision of finite element methods.

{ 
In our numerical experience, the ResNet has a better numerical performance than FNN. }
Hence, without especial explanation, we always use the ResNet of width $50$ and depth $6$ with an activation function $\sigma(x) = \max \{ x^3, 0 \}$. Neural networks are trained by Adam optimizer \cite{KB2014} with a learning rate $\eta = 1\text{e}-03$. The batch size is $512$ for all 1D examples and $1024$ for all 2D examples. Deep learning algorithms are implemented by Python 3.7 using PyTorch 1.0. and a single NVIDIA Quadro P6000 GPU card. All finite element methods are implemented in MATLAB 2018b in an Intel Core i7, 2.6GHz CPU on a personal laptop with a 32GB RAM.

{ 
Let us first summarize notations used in this section. Suppose $u$ is the exact solution to a problem and $u_k^h$ is the approximation evaluated by the Int-Deep framework in the $k$-th iteration in the second phase. We denote the absolute difference between the ground truth and the numerical solution as
\begin{equation*}
	 e_k^h := u - u_k^h,
\end{equation*}
which will be measured with different norms to obtain the accuracy of our framework.
After completing iterations, $u^h$ is used as the approximate solution by traditional iterative methods, and $e^h$ is denoted as its absolute difference. Similarly, $e^{DL}$ is denoted as the absolute difference between the exact solution and the approximate solution by deep learning. In eigenvalue problems, let $\lambda$, $\lambda^{DL}$ and $\lambda^{h}$ denote the target eigenvalue, the approximate one by deep learning, and the approximate one by traditional methods, respectively. $\#$Epoch means the number of epochs in the Adam for deep learning and $\#$K stands for the number of iterations in the traditional iterative methods in Int-Deep. We always apply the Courant element method to solve the weak form in Algorithms $1$ to $3$.

We also summarize the numerical examples in this section in Table \ref{tab: summary} below, which could help the reader to better understand the structure of extensive numerical examples.

\begin{table}[H]
    \small
	\centering
	\caption{  Summary of numerical examples.}
		\begin{tabular}{cc||cc}
		\hline
		  Phase I  &      & &      Phase II with an initial guess by Phase I   \\
		\hline
		  Example 5.1 &  Phase I, Section 3.1   &   Example 5.4        &   Phase II, Section 4.1   \\
		  Example 5.2 &  Phase I, Section 3.2  &   Example 5.5      &   Phase II, Section 4.2, linear eigenvalue problem   \\
		  Example 5.3 &  Phase I, Section 3.3  &   Example 5.6     &   Phase II, Section 4.2, nonlinear eigenvalue problem   \\
		\hline
	\end{tabular}
	\label{tab: summary}
\end{table}
}

\subsection{Phase I of Int-Deep: Deep Learning Methods}

Let us first provide numerical examples to illustrate the performance of the proposed variational formulations for deep learning in Section \ref{sec:I}.

\subsubsection{Linear PDEs}

\begin{example}
\label{ex:VE}
Consider the second-order linear elliptic equation with the Dirichlet boundary condition in one dimension:
\begin{equation*}
	\begin{cases}
	- \frac{\dd}{\dx}\left(p(x)\frac{\dd u}{\dx} \right) + x^2 u(x) = f(x), \quad  & x \in  (-1,1),
	\\
	\quad u = 0, \quad & x = -1 \, \mbox{\rm or} \,  1 ,
	\end{cases}
\end{equation*}
with $p(x) = 1+x^2$ and $f(x) = \pi^2(1+x^2)\sin(\pi x)-2\pi x\cos(\pi x) + x^2\sin(\pi x)$. The exact solution of this problem is
\[
	u(x) = \sin(\pi x).
\]
\end{example}

In this experiment, the variational formulation in \eqref{eq: solvable optim theta} with a penalty constant $\gamma=500$ is applied in the deep learning method. To evaluate the test error, we adopt a mesh size $h=\frac{2}{512}$ to generate the uniform triangulation $\mathcal{T}_h$ as the test locations. The test errors during training are summarized in Table \ref{tab:VE} and Figure \ref{fig:VE}.
{  From Table \ref{tab:VE} and Figure \ref{fig:VE}, we know the absolute discrete maximum error is reduced to order 1e-3 after 300 epochs. However, the error oscillates after about 1000 epochs and it is difficult to get a highly accurate solution when \#Epoch increases. }

\begin{table}[H]
	\begin{minipage}[t]{0.3\linewidth}
		\small
		\centering
	\caption{Example \ref{ex:VE}}
	\begin{tabular}{cccc}
		\hline
		$\#$Epoch & $\|e^{DL}\|_{0,\infty,h}$\\
		\hline
    300     & {1.16e-3} \\
		500     & {6.09e-4} \\
		1000    & {9.36e-4} \\
		5000    & {3.09e-3} \\
		10000   & {2.49e-3} \\
		\hline
	\end{tabular}
	\label{tab:VE}
	\end{minipage}
	\begin{minipage}[t]{0.3\linewidth}
		\small
		\centering
	\caption{Example \ref{ex:VI}.}
	\begin{tabular}{cccc}
		\hline
		$\#$Epoch & $\|e^{DL}\|_{0,\infty,h}$ \\
		\hline
		300     &  {6.55e-2}  \\
		500     &  {5.67e-2}  \\
		1000    &  {6.78e-3}  \\
		5000    &  {4.27e-3}  \\
		10000   &  {2.22e-3}  \\
		\hline
	\end{tabular}
	\label{tab:VI}
	\end{minipage}
		\begin{minipage}[t]{0.4\linewidth}
		\small
		\centering
		\caption{Example \ref{ex:eig}.}
    \begin{tabular}{ccc}
			\hline
			$\#$Epoch & $|\lambda - \lambda^{DL}|$ & $\| u - u^{DL}\|_{0,\infty,h}$ \\
			\hline
      		300    &  {1.40e-1} &  {5.06e-2} \\
			500    &  {2.31e-2} &  {1.37e-2}  \\
			1000   &  {1.71e-2} &  {1.66e-2}  \\
			5000   &  {4.85e-3} &  {1.04e-2}  \\
			10000  &  {3.20e-3} &  {1.45e-2}  \\
			\hline
		\end{tabular}
		\label{tab:eig}
	\end{minipage}
\end{table}

\begin{figure}[H]
	\begin{minipage}[H]{0.31\linewidth}
		\centering
		\includegraphics[width = 5 cm]{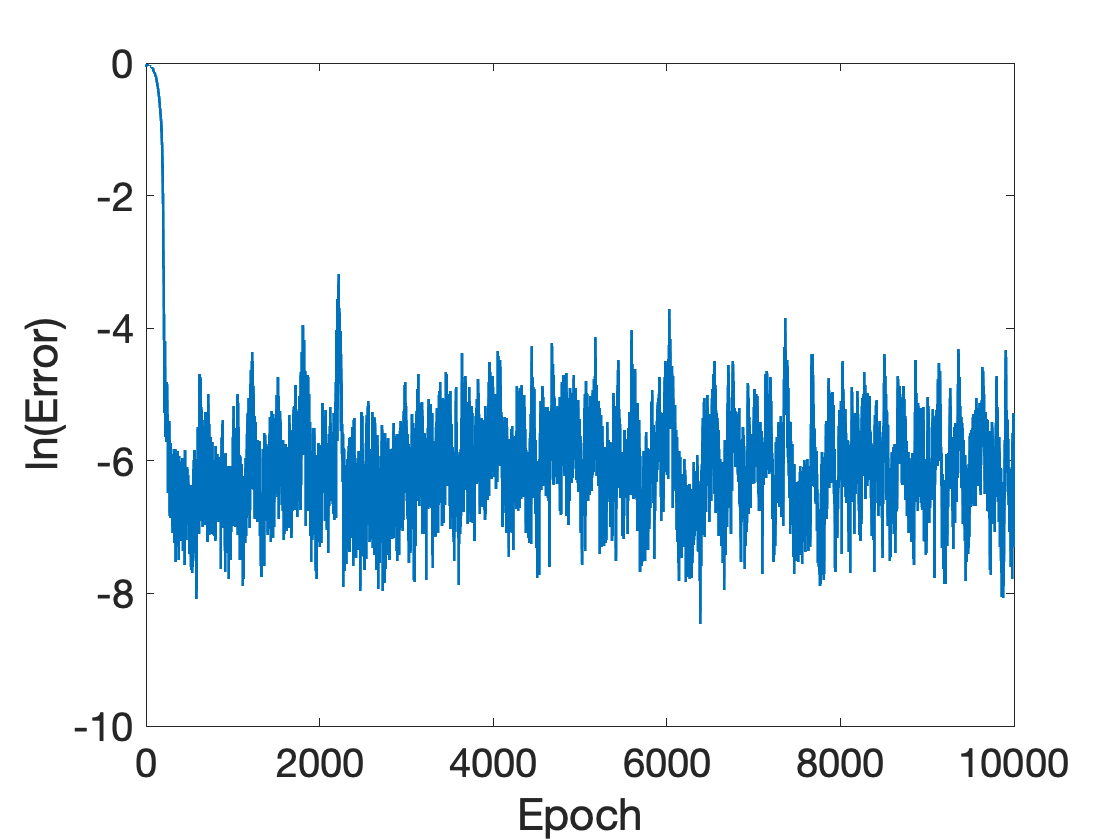}
		\caption{{ $e^{DL}$} of Example \ref{ex:VE}.}
		\label{fig:VE}
	\end{minipage}
	\begin{minipage}[H]{0.31\linewidth}
		\centering
		\includegraphics[width = 5 cm]{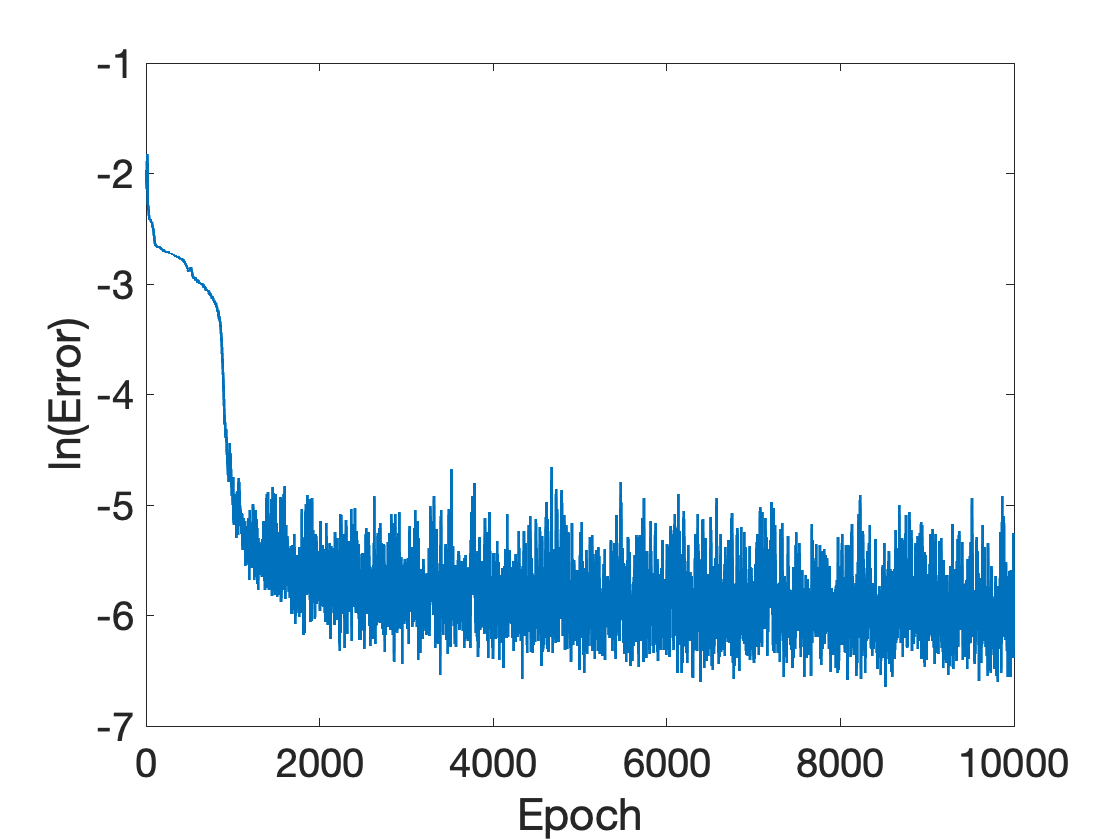}
		\caption{{ $e^{DL}$} of Example \ref{ex:VI}.}
		\label{fig:VI}
	\end{minipage}
		\begin{minipage}[H]{0.38\linewidth}
		\centering
		\includegraphics[width = 5 cm]{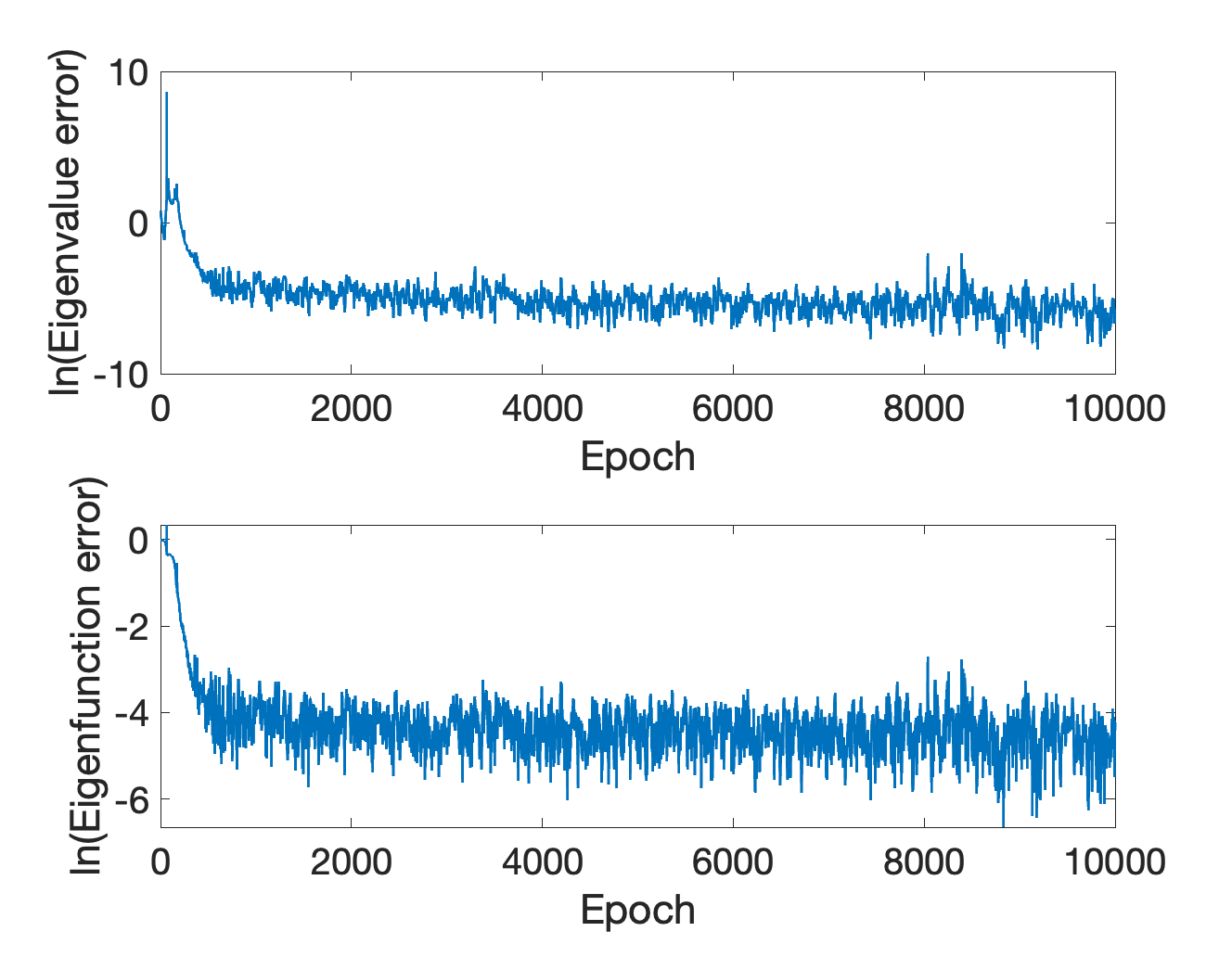}
		\caption{Test errors of Example \ref{ex:eig}.}
		\label{fig:eig}
	\end{minipage}
\end{figure}

\subsubsection{Variational Inequalities}

\begin{example}
\label{ex:VI}
Consider the simplified friction problem
\begin{equation*}
\label{ex: simplifiedfriction}
	\begin{cases}
		-\Delta u + u = f \quad & \mbox{\rm in} \ \Omega,
		\\
		|\partial _{\bm{n}}u| \leq g, \ u \partial_{\bm{n}}u + g|u| = 0 \quad & \mbox{\rm on} \ \Gamma_C,
		\\
		u = 0 \quad & \mbox{\rm on} \  \Gamma_D,
	\end{cases}	
\end{equation*}
where $\bm{n}$ is the outer normal vector, $\Omega = (0,1)^2$, $\Gamma_C = \{ 1 \} \times [0,1]$ and $\Gamma_D = \partial \Omega \setminus \Gamma_C$. We set $g=1$ and choose $f$ such that the problem has an exact solution $u(x, y) = (\sin x - x\sin 1) \sin(2\pi y)$.
\end{example}

In this experiment, the variational formula in \eqref{eq:VP2} with a penalty constant $\gamma = 500$ is applied in the deep learning method. To evaluate the test error, we adopt a mesh size $h=\frac{1}{128}$ to generate the uniform triangulation $\mathcal{T}_h$ as the test locations. The test errors during training are summarized in Table \ref{tab:VI} and Figure \ref{fig:VI}.
{  From Table \ref{tab:VI} and Figure \ref{fig:VI}, we know the absolute discrete maximum error is reduced to order 1e-3 after 1000 epochs. However, the error oscillates after 2000 epochs and it is difficult to get a highly accurate solution when \#Epoch increases.}

\subsubsection{Eigenvalue Problems}

\begin{example}
\label{ex:eig}
Consider the following problem
\begin{equation*}
	\begin{cases}
		- u^{\prime \prime} = \lambda u, \quad & x \in \ (0,1),
		\\
		u = 0, \quad & x = 0 \ \mbox{\rm or} \  1.
	\end{cases}
\end{equation*}
The smallest eigenvalue is $\pi^2$ and the corresponding eigenfunction is
$
	 u(x) = \sin(\pi(x-1)).
$ 		
\end{example}

In this experiment, the variational formula in \eqref{eq:VP3} with $x_0 = 0.5$ and $\gamma=100$ is applied in the deep learning method. To evaluate the test error, we adopt a mesh size $h=\frac{1}{256}$ to generate the uniform triangulation $\mathcal{T}_h$ as the test locations. The test errors during training are summarized in Table \ref{tab:eig} and Figure \ref{fig:eig}.
{  From Table \ref{tab:eig} and Figure \ref{fig:eig}, we know the absolute discrete maximum error of the eigenfunction is reduced to order 1e-2 after { 300} epochs and the eigenvalue error is reduced to order 1e-3 after 5000 epochs. Both errors oscillate after 1000 epochs and it is hard to to get high-accuracy as \#Epoch increases.}

\subsection{Phase II of Int-Deep: Traditional Iterative Methods}

Now we use the approximate solution by deep learning methods as the initial guess in traditional iterative methods. { We will} show that DNNs of size $O(1)$ and trained with $O(100)$ iterations are good enough for traditional iterative methods to converge in at most $O(\log(\frac{1}{\epsilon}))$ iterations to the $\epsilon$ precision of finite element methods.

\subsubsection{Semilinear PDEs}

We consider semilinear elliptic equations with homogeneous Dirichlet boundary conditions to demonstrate the efficiency of the Int-Deep framework in Algorithm~\ref{alg: semilinear} (deep learning combined with Newton's method for semilinear PDE) and to verify Theorem~\ref{thm: semilinear PDE}.

\begin{example}
\label{ex: nonlinear PDE}
Consider the following semilinear elliptic equations
\begin{equation*}
	\begin{cases}
	- \Delta u - (u-1)^3 +(u+2)^2 = f \quad  & \mbox{\rm in} \  \Omega ,
	\\
	\quad u = 0 \quad & \mbox{\rm on} \ \partial \Omega ,
	\end{cases}
\end{equation*}
where  $\Omega = (0,1)^d$ with $d=1$ or $2$. We choose $f$ such that the problem has an exact solution
$ u(x) = 3\sin(2\pi x)$ for $d=1$ and $ u(x) = 3\sin(2\pi x)\sin(2\pi y) $ for $d=2$.
\end{example}

{\bf Case $d=1$.}

First of all, we show that Newton's method cannot converge to a solution without a good initial guess $u_0$ and the deep learning method can provide a good initial guess. We apply Newton's iteration in Algorithm~\ref{alg: semilinear} with several types of initial guesses $u_0$ listed in Table \ref{tab: 1d Newton error}, and the other parameters are taken as $h=\frac{1}{1024}$, $N_{\text{max}} = 15$, and $\epsilon = 0.01 \times h^2$. $u^{DL}$ is obtained via the deep learning approach based on the variational formula   \eqref{eq: solvable optim theta} with $\gamma = 500$ and $200$ epochs in the Adam. Without good knowledge of the ground truth solution or $u^{DL}$, Newton's method fails to converge to a good numerical solution.

\begin{table}[H]
    \small
	\centering
	\caption{The performance of Newton's method with different initial guesses $u_0$.
	$\omega$ stands for a Gaussian random noise with mean zero and unit variance. }
		\begin{tabular}{cccc||cccc}
		\hline
		$u_0$       & $\#$K  &  $\|e_0^h\|_{0,\infty,h}$  & $\|e^h\|_{0,\infty,h}$  & $u_0$  & $\#$K  & $\|e_0^h\|_{0,\infty,h}$  &  $\|e^h\|_{0,\infty,h}$ \\
		\hline
		${1}$    & 5  &  {4.00e+0} &  {1.95e+0}  & $\omega$           & 6  &  {5.68e+0} &  {3.05e+0} \\
		${2}$    & 5  &  {5.00e+0} &  {1.95e+0}  &${1}+\omega$        & 5  &  {6.33e+0} &  {1.95e+0} \\
		${5}$    & 15 &  {8.00e+0} &  {1.28e+5}  &$-{1}+\omega$       & 15 &  {6.95e+0} &  {2.65e+6} \\
		${-1}$   & 15 &  {4.00e+0} &  {9.50e+4}  &$u+\omega$          & 6  &  {4.03e+0} &  {1.82e-5} \\
		${-2}$   & 12 &  {5.00e+0} &  {4.66e+0}  &$u+2.5\times\omega$ & 15 &  {8.89e+0} &  {2.10e+5} \\
		${-5}$   & 15 &  {8.00e+0} &  {6.45e+4}  & $u^{DL}$           & 5  &  {2.97e-1} &  {1.82e-5} \\
		\hline
	\end{tabular}
	\label{tab: 1d Newton error}
\end{table}

Besides, we also test the performance of the Picard's iteration in this example. We apply Picard's iteration with several types of initial guesses $u_0$ listed in Table \ref{tab: 1d Picard error}, and the other parameters are taken as $h=\frac{1}{1024}$,  $N_{\text{max}} = 15$, and $\epsilon = 0.01 \times h^2$. Picard's iteration fails to converge to the right solution in all tests.

\begin{table}[H]
    \small
	\centering
	\caption{The performance of Picard's method with different initial guesses $u_0$.
	$\omega$ stands for a Gaussian random noise with mean zero and unit variance. }
		\begin{tabular}{cccc||cccc}
		\hline
		$u_0$       & $\#$K  & $\|e_0^h\|_{0,\infty,h}$  &  $\|e^h\|_{0,\infty,h}$    & $u_0$  & $\#$K  &  $\|e_0^h\|_{0,\infty,h}$  &  $\|e^h\|_{0,\infty,h}$   \\
		\hline
		$\bm{1}$    & 15  &  {4.00e+0} &  {1.95e+0}  & $\omega$        & 15  &  {5.92e+0}  &  {1.95e+0} \\
		$\bm{0}$    & 15  &  {3.00e+0} &  {1.95e+0}  & $u+\omega$      & 10  &  {3.59e+0}  & nan \\
		\hline
	\end{tabular}
	\label{tab: 1d Picard error}
\end{table}

Furthermore, we show that the two-grid method also needs a good initial guess $u_0$ for Newton's method in the coarse grid stage; otherwise, the two-grid method cannot converge well. Denote the mesh size of the coarse grid as $H$ and the mesh size of the fine grid as $h$. The numerical results in Table \ref{tab: 1d semi two-grid} summarize the performance of the two-grid method with several types of initial guesses $u_0$ and different mesh sizes $H$ and $h$. The two grid method can converge to the correct solution only in the case when the initial guess of the coarse grid stage is provided by the deep learning method.

\begin{table}[H]
    \small
	\centering
	\caption{The performance of the two-grid method for Example \ref{ex: nonlinear PDE} in 1D with different initial guesses and mesh sizes.}
    \begin{tabular}{c|ccc|ccc|ccc}
		\hline
		{$u_0$}  & $H$ & $h$ & $\| e^h \|_{0,\infty,h}$  & $H$ & $h$ & $\| e^h \|_{0,\infty,h}$ & $H$ & $h$ & $\| e^h \|_{0,\infty,h}$ \\		
		\hline
		1 & $2^{-4}$ & $2^{-8}$ &  {1.95e+0} & $2^{-5}$ & $2^{-10}$ &  {1.95e+0} & $2^{-6}$ &  $2^{-12}$ &  {1.95e+0} \\
		-1 & $2^{-4}$ & $2^{-8}$ &  {2.25e+1} & $2^{-5}$ & $2^{-10}$ &  {2.74e+2} & $2^{-6}$ & $2^{-12}$ &   {2.99e+2} \\	
		0 & $2^{-4}$ & $2^{-8}$ &  {1.95e+0} & $2^{-5}$ & $2^{-10}$ &  {1.95e+0} & $2^{-6}$ & $2^{-12}$ &   {1.95e+0} \\
		$\omega$ & $2^{-4}$ & $2^{-8}$ &  {3.06e+0} & $2^{-5}$ & $2^{-10}$ &  {3.05e+0} & $2^{-6}$ & $2^{-12}$ &  {3.05e+0} \\
		$u^{DL}$ & $2^{-4}$ & $2^{-8}$ &  {3.27e-3} & $2^{-5}$ & $2^{-10}$ &  {1.98e-4} & $2^{-6}$ & $2^{-12}$ &  {1.23e-5}\\		
		\hline
	\end{tabular}
 	\label{tab: 1d semi two-grid}
\end{table}

Next, we show that the initial guess by the deep learning approach enables Newton's method to converge to a solution with the precision of finite element methods, i.e., the numerical convergence order in terms of $h$ defined by
{ 
\[
\text{order} := \log_2 \frac{\|e_k^h\|}{\|e_k^{h/2}\|} 
\]
is $2$ as proved by Theorem \ref{thm: semilinear PDE}, where $\|\cdot\|=\|\cdot \|_{0,\infty,h}$, the same one as used in the above table. This convention applies to other numerical examples.}
For the purpose of convenience, we choose $h = 2^{-\ell}$ for different integers $\ell$'s and let $e_{h^2} = h^2$ as the theoretical precision of finite element methods. We repeatedly apply the same deep learning method as in Table \ref{tab: 1d Newton error} to generate different initial guesses $u_0^h$ for different sizes $h$. Let $N_{\text{max}} = 15$ and $\epsilon = 0.01 \times h^2$ in Algorithm \ref{alg: semilinear}. Table \ref{tab: 1d scaling error}  summarizes the performance of Algorithm \ref{alg: semilinear} and numerical results verify the accuracy and the convergence order of the Int-Deep framework, even though the number of epochs in the training of deep learning is $O(100)$ and the initial error is very large.

\begin{table}[H]
    \small
	\centering
	\caption{The performance of Int-Deep in Algorithm~\ref{alg: semilinear} for Example \ref{ex: nonlinear PDE} in 1D with different mesh sizes $h$. }
	\begin{tabular}{ccccccc}
		\hline
		$h$          & $e_{h^2}$ & $\#$ Epoch &$\|e^{DL}\|_{0,\infty,h}$  & $\#$K &  $\|e^h\|_{0,\infty,h}$ & order \\
		\hline
		$2^{-7}$     &  {6.10e-5} & 200 &  {2.97e-1} & 5 &  {1.17e-3}  & - \\
		$2^{-8}$     &  {1.53e-5} & 200 &  {2.97e-1} & 5 &  {2.92e-4}  &  {2.00} \\
		$2^{-9}$     &  {3.81e-6} & 200 &  {2.97e-1} & 5 &  {7.29e-5}  &  {2.00} \\
		$2^{-10}$    &  {9.54e-7} & 200 &  {2.97e-1} & 5 &  {1.82e-5}  &  {2.00} \\
		$2^{-11}$    &  {2.38e-7} & 200 &  {2.97e-1} & 5 &  {4.56e-6}  &  {2.00} \\
		$2^{-12}$    &  {5.96e-8} & 200 &  {2.97e-1} & 5 &  {1.14e-6}  &  {2.00} \\
		$2^{-13}$    &  {1.49e-8} & 200 &  {2.97e-1} & 5 &  {2.85e-7}  &  {2.00} \\
		$2^{-14}$    &  {3.73e-9} & 200 &  {2.97e-1} & 5 &  {7.12e-8}  &  {2.00} \\
		\hline
	\end{tabular}
	\label{tab: 1d scaling error}
\end{table}

Finally, we show that the performance of the Int-Deep in Algorithm \ref{alg: semilinear} is independent of the size of DNNs, which is supported by the numerical results in Table \ref{tab: 1d differnet NN }.

\begin{table}[H]
    \small
  \centering
  \caption{The performance of Int-Deep in Algorithm~\ref{alg: semilinear} for Example \ref{ex: nonlinear PDE} in 1D with different DNN width $N$ and depth $L$ when $h=\frac{1}{1024}$ and $\#$Epoch = 400.}
    \begin{tabular}{c|ccc|ccc|ccc}
    \hline
    \multirow{2}{*}{\diagbox{L}{N}}  &\multicolumn{3}{c|}{10} &\multicolumn{3}{c|}{30} &\multicolumn{3}{c}{50}\\
    \cline{2-10}
    & ${\| e^{DL}\|_{0,\infty,h}}$  & $\|e^h\|_{0,\infty,h}$  & $\#$K  & ${\| e^{DL}\|_{0,\infty,h}}$ &   $\|e^h\|_{0,\infty,h}$  & $\#$K & ${\| e^{DL}\|_{0,\infty,h}}$  & $\|e^h\|_{0,\infty,h}$  & $\#$K   \\
    \hline
    2 &  {3.72e+0} &  {1.82e-5} & 8 &  {2.27e-1} &  {1.82e-5} & 5 &  {5.19e-2} &  {1.82e-5}& 4  \\
    4 &  {1.05e-1} &  {1.82e-5} & 4 &  {1.63e-2} &  {1.82e-5} & 3 &  {1.51e-2} &  {1.82e-5} & 3  \\
    6 &  {9.84e-2} &  {1.82e-5} & 4 &  {8.48e-3} &  {1.82e-5} & 3 &  {5.34e-3} &  {1.82e-5} & 3  \\
    \hline
  \end{tabular}
  \label{tab: 1d differnet NN }
\end{table}

{\bf Case $d=2$.}

Again, we show that the initial guess by the deep learning approach enables Newton's method to converge to a solution with the precision of finite element methods, i.e., the numerical convergence order in terms of $h$ is almost $2$ as proved by Theorem \ref{thm: semilinear PDE}. We repeatedly apply the variational formula in  \eqref{eq: solvable optim theta} with $\gamma = 500$ and the deep learning method to generate different initial guesses $u_0^h$ for different sizes $h$. Let $N_{\text{max}} = 15$ and $\epsilon = 0.01 \times h^2$ in Algorithm \ref{alg: semilinear}. {  Let $e_{h^2} = |h^2\log h|$.} Table \ref{tab: 2d scaling error}   summarizes the performance of Algorithm \ref{alg: semilinear} and numerical results verify the accuracy and the convergence order of the Int-Deep framework, even though the number of epochs in deep learning is $O(100)$ and the initial error is very large.

\begin{table}[H]
	\centering
	\caption{ The performance of Int-Deep in Algorithm~\ref{alg: semilinear} for Example \ref{ex: nonlinear PDE} in 2D with different mesh sizes $h$. }
    \begin{tabular}{ccccccc}
		\hline
		$h$          & $e_{h^2}$ & $\#$ Epoch &  $\| e^{DL}\|_{0,\infty,h}$  & $\#$K & $\| e^h\|_{0,\infty,h}$  & order \\
		\hline
		$2^{-4}$    &  {1.08e-2} & 200 &  {2.19e+0} & 5 &  {2.03e-1} & - \\
		$2^{-5}$    &  {3.38e-3} & 200 &  {2.19e+0} & 5 &  {5.92e-2} &  {1.78} \\
		$2^{-6}$    &  {1.02e-3} & 200 &  {2.20e+0} & 5 &  {1.55e-2} &  {1.94} \\
		$2^{-7}$    &  {2.96e-4} & 200 &  {2.20e+0} & 6 &  {3.92e-3} &  {1.98} \\
		$2^{-8}$    &  {8.46e-5} & 200 &  {2.20e+0} & 6 &  {9.83e-4} &  {2.00} \\
		\hline
	\end{tabular}
	\label{tab: 2d scaling error}
\end{table}

Finally, we show that the performance of the Int-Deep in Algorithm \ref{alg: semilinear} is independent of the size of DNNs, which is supported by the numerical results in Table \ref{tab: 2d differnet NN}.

\begin{table}[H]
    \small
  \centering
  \caption{The performance of Int-Deep in Algorithm~\ref{alg: semilinear} for Example \ref{ex: nonlinear PDE} in 2D with different DNN width $N$ and depth $L$ when $h=\frac{1}{128}$ and $\#$Epoch = 400. }
  \begin{tabular}{c|ccc|ccc|ccc}
    \hline
    \multirow{2}{*}{\diagbox{L}{N}} &\multicolumn{3}{c|}{10} &\multicolumn{3}{c|}{30} &\multicolumn{3}{c}{50}\\
    \cline{2-10}
    &$ {\| e^{DL}\|_{0,\infty,h}}$  &  $\|e^h\|_{0,\infty,h}$  & $\#$K  & ${\| e^{DL}\|_{0,\infty,h}}$  &   $\|e^h\|_{0,\infty,h}$  & $\#$K & ${\| e^{DL}\|_{0,\infty,h}}$  &   $\|e^h\|_{0,\infty,h}$  & $\#$K \\
    \hline
    2 &  {3.15e+0} &  {3.92e-3} & 6  &  {2.91e+0} &  {3.92e-3} & 6  &  {3.32e+0} &  {3.92e-3} & 6  \\
    4 &  {4.06e+0} &  {3.92e-3} & 6  &  {2.53e+0} &  {3.92e-3} & 5  &  {7.59e-1} &  {3.92e-3} & 4  \\
    6 &  {3.54e+0} &  {3.92e-3} & 6  &  {1.20e+0} &  {3.92e-3} & 4  &  {1.05e+0} &  {3.92e-3} & 4  \\
    \hline
  \end{tabular}
  \label{tab: 2d differnet NN}
\end{table}

\subsubsection{Linear eigenvalue problems}

Here, we demonstrate the efficiency of the Int-Deep framework in Algorithm~\ref{alg: eigen} (deep learning combined with the power method) and verify Theorem~\ref{thm: eigenvalue problem}.

\begin{example}
\label{ex: 2d eigen problem}
Consider the following problem
\begin{equation*}
	\begin{cases}
		- \Delta u = \lambda u \quad   &\mbox{\rm in}  \ \Omega,
		\\
		\quad u = 0 \quad &\mbox{\rm on} \ \partial \Omega,
	\end{cases}
\end{equation*}
where $\Omega = (0,1)^d$ for $d=1$ and $2$. The smallest eigenvalue is $d\pi^2$. The corresponding eigenfunction is $u(x) = \sin(\pi(x-1))$ for $d=1$, and $u(x,y) = \sin(\pi(x-1))\sin(\pi(y-1))$ for $d=2$.
\end{example}

{\bf Case $d=1$.}

We show that the initial guess by the deep learning approach enables the power method to quickly converge to a solution with the precision of finite element methods, i.e., the numerical convergence order in terms of $h$ is
$2$ as proved by Theorem \ref{thm: eigenvalue problem}.
We repeatedly apply the variational formula in \eqref{eq:VP3} (with $x_0 = 0.5$ and $\gamma = 100$) and the deep learning method to generate different initial guesses $u_0^h$ for different sizes $h$. Let $N_{\text{max}} = 10$ and $\epsilon = h^2$ in Algorithm \ref{alg: eigen}. Table \ref{tab: 1d eigen error order}  summarizes the performance of Algorithm \ref{alg: eigen} and numerical results verify the accuracy and the convergence order of the Int-Deep framework, even though the number of epochs in deep learning is $O(100)$.

\begin{table}[H]
	\centering
	\caption{The performance of Int-Deep in Algorithm~\ref{alg: eigen} for Example \ref{ex: 2d eigen problem} in 1D with different mesh sizes $h$.}
	\begin{tabular}{cccccccc}
		\hline
		$h$ & $\#$Epoch & { $\| e^{h}_0\|_{0}$} & $\#$K & $|\lambda - \lambda^h|$ & eigenvalue order & { $\| e^h\|_{0}$}  & eigenfunction order\\
		\hline
		$2^{-5}$ & 300 &  {1.13e-2}  & 3  &  {7.93e-3}  & -     &  {8.08e-4}  &  -      \\
		$2^{-6}$ & 300 &  {1.14e-2}  & 4  &  {1.98e-3}  &  {2.00}   &  {2.02e-4}  &  {2.00}  \\
		$2^{-7}$ & 300 &  {1.14e-2}  & 5  &  {4.95e-4}  &  {2.00}   &  {5.05e-5}  &  {2.00}  \\
		$2^{-8}$ & 300 &  {1.14e-2}  & 6  &  {1.24e-4}  &  {2.00}   &  {1.26e-5}  &  {2.00}  \\
		$2^{-9}$ & 300 &  {1.14e-2}  & 7  &  {3.04e-5}  &  {2.03}   &  {3.17e-6}  &  {2.00}  \\
		\hline
	\end{tabular}
	\label{tab: 1d eigen error order}
\end{table}

Next, we show that the performance of Int-Deep in Algorithm \ref{alg: eigen} is independent of the size of DNNs, which is supported by the numerical results in Table \ref{tab: 1d differnet NN eigenvalue}.

\begin{table}[H]
    \small
  \centering
  \caption{The performance of Int-Deep in Algorithm~\ref{alg: eigen} for Example \ref{ex: 2d eigen problem} in 1D with different DNN width $N$ and depth $L$ when $h=\frac{1}{512}$ and $\#$Epoch = 400. }
  \begin{tabular}{c|ccc|ccc|ccc}
    \hline
    \multirow{2}{*}{\diagbox{L}{N}} &\multicolumn{3}{c|}{10} &\multicolumn{3}{c|}{30} &\multicolumn{3}{c}{50}\\
    \cline{2-10}
    & $|\lambda-\lambda^h|$ &  $\| e^{h}\|_{0}$  & $\#$K & $|\lambda-\lambda^h|$ & $\| e^{h}\|_{0}$  & $\#$K & $|\lambda-\lambda^h|$ &  $\| e^{h}\|_{0}$  & $\#$K  \\
    \hline
    2 &  {3.06e-5} &  {3.30e-6} & 7 &  {2.61e-5} &  {3.40e-6} & 8 &  {3.67e-5} &  {3.47e-6} & 8 \\
    4 &  {3.10e-5} &  {3.26e-6} & 6 &  {3.09e-5} &  {3.17e-6} & 6 &  {3.24e-5} &  {3.19e-6} & 8 \\
    6 &  {4.25e-5} &  {4.42e-6} & 8 &  {2.98e-5} &  {3.24e-6} & 7 &  {3.09e-5} &  {3.30e-6} & 6 \\
    \hline
  \end{tabular}
  \label{tab: 1d differnet NN eigenvalue}
\end{table}

{\bf Case $d=2$.}

Again, we show that the initial guess by the deep learning approach enables the power method to quickly converge to a solution with the precision of finite element methods, i.e., the numerical convergence order in terms of $h$ is
$2$ as proved by Theorem \ref{thm: eigenvalue problem}.
We repeatedly apply the variational formula in \eqref{eq:VP3} (with $\bm{x}_0 = (0.5,0.5)$ and $\gamma = 100$) and the deep learning method to generate different initial guesses $u_0^h$ for different sizes $h$. Let $N_{\text{max}} = 10$ and $\epsilon = h^2$ in Algorithm \ref{alg: eigen}. Table \ref{tab: eigen error order}   summarizes the performance of Algorithm \ref{alg: eigen} and numerical results verify the accuracy and the convergence order of the Int-Deep framework, even though the number of epochs in deep learning is $O(100)$.

\begin{table}[H]
	\centering
	\caption{The performance of Int-Deep in Algorithm~\ref{alg: eigen} for Example \ref{ex: 2d eigen problem} in 2D with different mesh sizes $h$.   }
	\begin{tabular}{cccccccc}
		\hline
		$h$ & \#Epoch & { $\| e^{h}_0\|_{0}$} & \#K & $|\lambda - \lambda^h|$ & eigenvalue order & $\| e^{h}\|_{0}$ & eigenfunction order\\
		\hline
		$2^{-4}$  & 300 &  {5.56e-2} & 4  &  {1.91e-1}  & -     &  {6.63e-3}  &  -   \\
		$2^{-5}$  & 300 &  {5.65e-2} & 5  &  {4.76e-2}   &  {2.00}   &  {1.70e-3}  &  {1.96}  \\
		$2^{-6}$  & 300 &  {5.68e-2} & 7  &  {1.19e-2}  &  {2.00}   &  {4.19e-4}  &  {2.02}  \\
		$2^{-7}$  & 300 &  {5.69e-2} & 8  &  {2.97e-3}  &  {2.00}   &  {1.07e-4}  &  {1.97}  \\
		$2^{-8}$  & 300 &  {5.69e-2} & 10 &  {7.43e-4}  &  {2.00}   &  {2.62e-5}  &  {2.03}  \\
		\hline
	\end{tabular}
	\label{tab: eigen error order}
\end{table}

Next, we show that the performance of Int-Deep in Algorithm \ref{alg: eigen} is independent of the size of DNNs, which is supported by the numerical results in Table \ref{tab: 2d differnet NN eigenvalue}.

\begin{table}[H]
    \small
  \centering
  \caption{The performance of Int-Deep in Algorithm~\ref{alg: eigen} for Example \ref{ex: 2d eigen problem} in 2D with different DNN width $N$ and depth $L$ when $h=\frac{1}{128}$ and $\#$Epoch = 400.  }
  \begin{tabular}{c|ccc|ccc|ccc}
    \hline
    \multirow{2}{*}{\diagbox{L}{N}} &\multicolumn{3}{c|}{10} &\multicolumn{3}{c|}{30} &\multicolumn{3}{c}{50}\\
    \cline{2-10}
    & $|\lambda-\lambda^h|$ &  $\| e^{h}\|_{0}$ & $\#$K  & $|\lambda-\lambda^h|$ & $\| e^{h}\|_{0}$ & $\#$K  & $|\lambda-\lambda^h|$ & $\| e^{h}\|_{0}$ & $\#$K   \\
    \hline
    2 &  {2.97e-3} &  {1.09e-4} & 9  &  {2.97e-3} &  {1.08e-4} & 8  &  {2.97e-3} &  {1.08e-4} & 8  \\
    4 &  {2.97e-3} &  {1.07e-4} & 7  &  {2.97e-3} &  {1.07e-4} & 7  &  {2.97e-3} &  {1.09e-4} & 7  \\
    6 &  {2.97e-3} &  {1.07e-4} & 6  &  {2.97e-3} &  {1.11e-4} & 7  &  {2.97e-3} &  {1.05e-4} & 8  \\
    \hline
  \end{tabular}
  \label{tab: 2d differnet NN eigenvalue}
\end{table}

\subsubsection{Nonlinear eigenvalue problems}
As used in \cite{GaoYangMeza2009}, denote by $res^{DL}$ and $res^h$ the residuals of $u^{DL}$ and $u^h$ corresponding to the matrix form of \eqref{nonlinear eig variational2}, respectively. {  In this example, we denote $\| \cdot \|_2$ as the vector $L_2-$norm in Euclidian space.} Then we investigate the accuracy of { Int-Deep} framework in Algorithm~\ref{alg: nonlinear eig} through the case of the nonlinear Schr\"{o}dinger equation.

\begin{example}\label{ex:nonlinear eigen}
Consider the following problem	
\begin{equation*}\label{ex: GP}
	\begin{cases}
		- \Delta u + V(x)u + 10 u^3= \lambda u &\quad \mbox{\rm in} \ \Omega,
		\\
		u =0 &\quad \mbox{\rm on} \ \partial\Omega,
		\\
		\| u \|_0 =1,
	\end{cases}
\end{equation*}
where $\Omega = (0,1)^d$ for $d=1,2$, $V(x)$ is a potential with two Gaussian wells on $(0,1)$ for $d=1$,  and with four Gaussian wells on $(0,1)^2$ for $d=2$.
\end{example}

\textbf{Case $d=1$}

 We first show that the initial guess by the deep learning  approach enables Newton's method to quickly converge to the target solution. We repeatedly apply the variational formula in \eqref{GP variational 3} (with $x_0 = 0.5$ and $\gamma = 100$) and the deep learning method to generate different initial guesses $u_0^h$ for different sizes $h$. Let $N_{\text{max}} = 10$ and $\epsilon = h^2$ in Algorithm \ref{alg: nonlinear eig}. Table \ref{tab: 1d nonlinear eigen error}  summarizes the performance of Algorithm \ref{alg: nonlinear eig}.

\begin{table}[H]
	\centering
	\caption{The performance of Int-Deep in Algorithm~\ref{alg: nonlinear eigen} for Example \ref{ex:nonlinear eigen} in 1D with different mesh sizes $h$.}
	\begin{tabular}{cccccccc}
		\hline
		$h$ & $\#$Epoch & ${\|res^{DL}\|_2}$ & $\lambda^{DL}$ & $\#$K & ${\|res^h\|_2}$ & $\lambda^h$ \\
		\hline
		$2^{-5}$ & 300 &  {6.10e-1} &  {23.71} & 2 &  {2.41e-6}  &  {23.60} \\
		$2^{-6}$ & 300 &  {4.72e-1} &  {23.69} & 3 &  {1.13e-10} &  {23.58} \\
		$2^{-7}$ & 300 &  {3.46e-1} &  {23.69} & 3 &  {1.54e-11} &  {23.58} \\
		$2^{-8}$ & 300 &  {2.49e-1} &  {23.69} & 3 &  {2.57e-12} &  {23.58} \\
		$2^{-9}$ & 300 &  {1.77e-1} &  {23.69} & 3 &  {1.51e-12} &  {23.58} \\
		\hline
	\end{tabular}
	\label{tab: 1d nonlinear eigen error}
\end{table}

Next, we show that the performance of Int-Deep in Algorithm \ref{alg: nonlinear eigen} is independent of the size of DNNs, which is supported by the numerical results in Table \ref{tab: 1d differnet NN performance GP}.

\begin{table}[H]
  \small
  \centering
  \caption{ The performance of Int-Deep in Algorithm~\ref{alg: nonlinear eig} for Example \ref{ex:nonlinear eigen} in 1D with different DNN width $N$ and depth $L$ when $h=\frac{1}{512}$ and $\#$Epoch = 400. }
  \begin{tabular}{c|ccc|ccc|ccc}
    \hline
    \multirow{2}{*}{\diagbox{L}{N}} &\multicolumn{3}{c|}{10} &\multicolumn{3}{c|}{30} &\multicolumn{3}{c}{50}\\
    \cline{2-10}
    & $\lambda^h$ & ${\|res^h\|_2}$ & $\#$K  & $\lambda^h$& ${\|res^h\|_2}$ & $\#$K & $\lambda^h$& ${\|res^h\|_2}$ & $\#$K    \\
    \hline
    2 &  {23.58} &  {1.33e-12} & 5  &  {23.58} &  {1.44e-12} & 4 & {23.58} &  {1.50e-12} & 3  \\
    4 &  {23.58} &  {1.53e-12} & 3  &  {23.58} &  {2.14e-12} & 3 & {23.58} &  {1.48e-12} & 3  \\
    6 &  {23.58} &  {1.44e-12} & 3  &  {23.58} &  {1.49e-12} & 3 & {23.58} &  {1.34e-12} & 3  \\
    \hline
  \end{tabular}
  \label{tab: 1d differnet NN performance GP}
\end{table}

{\bf Case $d=2$.}

Again, we show that the initial guess by the deep learning approach enables Newton's method to quickly converge to the target solution. We repeatedly apply the variational formula in \eqref{GP variational 3} (with $\bm{x}_0 = (0.5,0.5)$ and $\gamma = 100$) and the deep learning method to generate different initial guesses $u_0^h$ for different sizes $h$. Let $N_{\text{max}} = 10$ and $\epsilon = h^2$ in Algorithm \ref{alg: nonlinear eig}. Table \ref{tab: 2d nonlinear eigen error}   summarizes the performance of Algorithm \ref{alg: nonlinear eig}.

\begin{table}[H]
	\centering
	\caption{The performance of Int-Deep in Algorithm~\ref{alg: nonlinear eigen} for Example \ref{ex:nonlinear eigen} in 2D with different mesh sizes $h$.}
	\begin{tabular}{cccccccc}
		\hline
		$h$ & $\#${Epoch} & ${\|res^{DL}\|_2}$ & $\lambda^{DL}$ & $\#$K & ${\|res^h\|_2}$ & $\lambda^h$ \\
		\hline
		$2^{-4}$ & 300 &  {4.07e-1} &  {39.74} & 2 &  {9.06e-6}  &  {39.46} \\
		$2^{-5}$ & 300 &  {2.46e-1} &  {39.49} & 2 &  {2.13e-6}  &  {39.19} \\
		$2^{-6}$ & 300 &  {1.35e-1} &  {39.43} & 3 &  {1.43e-10}  &  {39.12} \\
		$2^{-7}$ & 300 &  {7.04e-2} &  {39.42} & 3 &  {1.53e-11} &  {39.10} \\
		\hline
	\end{tabular}
	\label{tab: 2d nonlinear eigen error}
\end{table}

Next, we show that the performance of Int-Deep in Algorithm \ref{alg: nonlinear eigen} is independent of the size of DNNs, which is supported by the numerical results in Table \ref{tab: 2d differnet NN performance GP}.

\begin{table}[H]
  \small
  \centering
  \caption{ The performance of Int-Deep in Algorithm~\ref{alg: nonlinear eig} for Example \ref{ex:nonlinear eigen} in 2D with different DNN width $N$ and depth $L$ when $h=\frac{1}{128}$ and $\#$Epoch = 400. }
  \begin{tabular}{c|ccc|ccc|ccc}
    \hline
    \multirow{2}{*}{\diagbox{L}{N}} &\multicolumn{3}{c|}{10} &\multicolumn{3}{c|}{30} &\multicolumn{3}{c}{50}\\
    \cline{2-10}
    & $\lambda^h$ & ${\|res^h\|_2}$ & $\#$K  & $\lambda^h$& ${\|res^h\|_2}$ & $\#$K & $\lambda^h$& ${\|res^h\|_2}$ & $\#$K   \\
    \hline
    2 &  {39.10} &  {1.60e-12} & 3  &  {39.10} &  {7.88e-13} & 4  &  {39.10} &  {4.26e-13} & 3  \\
    4 &  {39.10} &  {3.28e-10} & 3  &  {39.10} &  {1.30e-11} & 3  &  {39.10} &  {6.33e-12} & 3  \\
    6 &  {39.10} &  {3.77e-11} & 3  &  {39.10} &  {8.30e-11} & 3  &  {39.10} &  {2.10e-10} & 3  \\
    \hline
  \end{tabular}
  \label{tab: 2d differnet NN performance GP}
\end{table}

\section{Conclusion}\label{sec:con}

This paper proposed the Int-Deep framework from a new point of view for designing highly efficient solvers of low-dimensional nonlinear PDEs with a finite element accuracy leveraging both the advantages of traditional algorithms and deep learning approaches.  The Int-Deep framework consists of two phases. In the first phase, an approximate solution to the given nonlinear PDE is obtained via deep learning approaches using DNNs of size $O(1)$ and $O(100)$ iterations. In the second phase, the approximate solution provided by deep learning can serve as a good initial guess such that traditional iterative methods converge in $O(\log(\frac{1}{\epsilon}))$ iterations to the $\epsilon$ precision of finite element methods. The Int-Deep framework outperforms existing purely deep learning-based methods or traditional iterative methods. The code can be shared per request.

In particular, based on variational principles, we propose new methods to formulate the problem of solving nonlinear PDEs into an unconstrained minimization problem of an expectation over a function space parametrized via DNNs, which can be solved efficiently via batch stochastic gradient descent (SGD) methods due to the special form of expectation. Unlike previous methods in which the form of expectation is only derived for nonlinear PDEs related to variational equations, our proposed method can also handle variational inequalities and eigenvalue problems, providing a unified variational framework for a wider range of nonlinear problems. With the good initialization given by deep learning, we hope to reduce the difficulty of designing an efficient traditional iterative algorithm for variational inequalities. We will leave this as future work.

{\bf Acknowledgments.} J. H. was partially supported by NSFC (Grant No.11571237). H. Y. was partially supported by the National Science Foundation under the grant award 1945029. We would like to thank reviewers for their valuable comments to improve the manuscript.

\bibliographystyle{plain}

\begin{thebibliography}{}

\end{thebibliography}


\begin{thebibliography}{100}

\bibitem{Adams1975}
R.A. Adams.
\newblock {\em Sobolev Spaces}.
\newblock Academic Press, New York-London, 1975.

\bibitem{AmbrosettiMalchiodi2007}
A.~Ambrosetti and A.~Malchiodi.
\newblock {\em Nonlinear Analysis and Semilinear Elliptic Problems}.
\newblock Cambridge University Press, Cambridge, 2007.

\bibitem{nature}
J.R. Anglin and W.~Ketterle.
\newblock Bose -{E}instein {C}ondensation of {A}tomic {G}ases.
\newblock {\em Nature}, 416:211--218, 2002.

\bibitem{BabuskaOsborn1991}
I.~Babu\v{s}ka and J.~Osborn.
\newblock Eigenvalue {P}roblems.
\newblock In {\em Handbook of Numerical Analysis}. North-Holland, Amsterdam,
  1991.

\bibitem{BabuskaOsborn1989}
I.~Babu\v{s}ka and J.~E. Osborn.
\newblock Finite {E}lement-{G}alerkin {A}pproximation of the {E}igenvalues and
  {E}igenvectors of {S}elfadjoint {P}roblems.
\newblock {\em Math. Comp.}, 52:275--297, 1989.

\bibitem{barron1993}
A.R. Barron.
\newblock Universal {A}pproximation {B}ounds for {S}uperpositions of a
  {S}igmoidal {F}unction.
\newblock {\em IEEE Trans. Inform. Theory}, 39:930--945, 1993.

\bibitem{Bartels}
S.~Bartels.
\newblock {\em Numerical Methods for Nonlinear Partial Differential Equations}.
\newblock Springer, Cham, 2015.

\bibitem{BASDEVANT198623}
C.~Basdevant, M.~Deville, P.~Haldenwang, J.M. Lacroix, J.~Ouazzani, R.~Peyret,
  P.~Orlandi, and A.T Patera.
\newblock {S}pectral and {F}inite {D}ifference {S}olutions of the {B}urgers
  {E}quation.
\newblock {\em Comput. {$\&$} Fluids}, 14:23 -- 41, 1986.

\bibitem{BERG2018}
J.~Berg and K.~Nystr{\"o}m.
\newblock {A} {U}nified {D}eep {A}rtificial {N}eural {N}etwork {A}pproach to
  {P}artial {D}ifferential {E}quations in {C}omplex {G}eometries.
\newblock {\em Neurocomputing}, 317:28 -- 41, 2018.

\bibitem{BochevGunzburger2009}
P.~B. Bochev and M.D. Gunzburger.
\newblock {\em Least-squares Finite Element Methods}.
\newblock Springer, New York, 2009.

\bibitem{BrennerScott2008}
S.C. Brenner and L.R. Scott.
\newblock {\em The Mathematical Theory of Finite Element Methods}.
\newblock Springer, New York, 1994.

\bibitem{CancesChakirHeMaday}
E.~Canc\`{e}s, R.~Chakir, L.~He, and Y.~Maday.
\newblock {Two-grid {M}ethods for a {C}lass of {N}onlinear {E}lliptic
  {E}igenvalue {P}roblems}.
\newblock {\em IMA J. Numer. Anal.}, 38:605--645, 2018.

\bibitem{Cances2010}
E.~Canc\`{e}s, R.~Chakir, and Y.~Maday.
\newblock Numerical {A}nalysis of {N}onlinear {E}igenvalue {P}roblems.
\newblock {\em J. Sci. Comput.}, 45:90--117, 2010.

\bibitem{Carleo2017}
G.~Carleo and M.~Troyer.
\newblock Solving the {Q}uantum {M}any-body {P}roblem with {A}rtificial
  {N}eural {N}etworks.
\newblock {\em Science}, 355:602--606, 2017.

\bibitem{10.2307/2004575}
A.~J. Chorin.
\newblock Numerical {S}olution of the {N}avier-{S}tokes {E}quations.
\newblock {\em Math. Comp.}, 22:745--762, 1968.

\bibitem{Ciarlet1978}
P.G. Ciarlet.
\newblock {\em The Finite Element Method for Elliptic Problems}.
\newblock North-Holland, Amsterdam-New York-Oxford, 1978.

\bibitem{doi:10.1002/cnm.1640100303}
M.W.M.~G. Dissanayake and N.~Phan-Thien.
\newblock Neural-network-based {A}pproximations for {S}olving {P}artial
  {D}ifferential {E}quations.
\newblock {\em Comm. Numer. Methods Engrg.}, 10:195--201, 1994.

\bibitem{DuvautLions1976}
G.~Duvaut and J.{-L}. Lions.
\newblock {\em Inequalities in Mechanics and Physics}.
\newblock Springer, Berlin-New York, 1976.

\bibitem{EYU2018}
W.~E and B.~Yu.
\newblock The {D}eep {R}itz {M}ethod: a {D}eep {L}earning-based {N}umerical
  {A}lgorithm for {S}olving {V}ariational {P}roblems.
\newblock {\em Commun. Math. Stat.}, 6:1--12, 2018.

\bibitem{Fan2019}
Y.~Fan, J.~Feliu-Fab\`{a}, L.~Lin, L.~Ying, and L.~Zepeda-N\'{u}{\~{n}}ez.
\newblock A {M}ultiscale {N}eural {N}etwork {B}ased on {H}ierarchical {N}ested
  {B}ases.
\newblock {\em Res. Math. Sci.}, 6:21--28, 2019.

\bibitem{Yuwei2018}
Y.~Fan, L.~Lin, L.~Ying, and L.~Zepeda-N\'{u}{\~{n}}ez.
\newblock A {M}ultiscale {N}eural {N}etwork {B}ased on {H}ierarchical
  {M}atrices.
\newblock {\em Multiscale Model. Simul.}, 17:1189--1213, 2019.

\bibitem{Fan2019BCRNetAN}
Y.~Fan, C.~Orozco~Bohorquez, and L.~Ying.
\newblock {BCR}-{N}et: a {N}eural {N}etwork {B}ased on the {N}onstandard
  {W}avelet {F}orm.
\newblock {\em J. Comput. Phys.}, 384:1--15, 2019.

\bibitem{PhysRev1}
A.~L. Fetter.
\newblock Vortices in an {I}mperfect {B}ose {G}as. {I}. {T}he {C}ondensate.
\newblock {\em Phys. Rev. (2)}, 138:A429--A437, 1965.

\bibitem{GaoYangMeza2009}
W.~Gao, C.~Yang, and J.~Meza.
\newblock Solving a {C}lass of {N}onlinear {E}igenvalue {P}roblems by
  {N}ewton's {M}ethod.
\newblock {\em preprint}, 2009.

\bibitem{Glowinski1984}
R.~Glowinski.
\newblock {\em Numerical Methods for Nonlinear Variational Problems}.
\newblock Springer, New York, 1984.

\bibitem{IanYoshuaAaron2016}
I.~Goodfellow, Y.~Bengio, and A.~Courville.
\newblock {\em Deep Learning}.
\newblock MIT Press, Cambridge, 2016.

\bibitem{GuYangZhou2020}
Y.~{Gu}, H.~{Yang}, and C.~{Zhou}.
\newblock {Select{N}et: {S}elf-paced {L}earning for {H}igh-dimensional
  {P}artial {D}ifferential {E}quations}.
\newblock {\em arXiv e-prints}, page arXiv:2001.04860, 2020.

\bibitem{Han8505}
J.~Han, A.~Jentzen, and W.~E.
\newblock Solving {H}igh-dimensional {P}artial {D}ifferential {E}quations
  {U}sing {D}eep {L}earning.
\newblock {\em Proc. Natl. Acad. Sci. USA}, 115:8505--8510, 2018.

\bibitem{HeZhangRenSun2016}
K.~{He}, X.~{Zhang}, S.~{Ren}, and J.~{Sun}.
\newblock Deep {R}esidual {L}earning for {I}mage {R}ecognition.
\newblock In {\em 2016 IEEE Conference on Computer Vision and Pattern
  Recognition (CVPR)}, pages 770--778, 2016.

\bibitem{PhysRev3}
P.~Hohenberg and W.~Kohn.
\newblock Inhomogeneous {E}lectron {G}as.
\newblock {\em Phys. Rev. (2)}, 136:B864--B871, 1964.

\bibitem{HYMAN1986113}
J.M. Hyman and B.~Nicolaenko.
\newblock The {K}uramoto-{S}ivashinsky {E}quation: a {B}ridge between {PDE}'s
  and {D}ynamical {S}ystems.
\newblock {\em Phys. D}, 18:113 -- 126, 1986.

\bibitem{Yuehaw2017}
Y.~Khoo, J.~Lu, and L.~Ying.
\newblock {Solving {P}arametric {PDE} {P}roblems with {A}rtificial {N}eural
  {N}etworks}.
\newblock {\em arXiv e-prints}, page arXiv:1707.03351, 2017.

\bibitem{Khoo2018}
Y.~Khoo, J.~Lu, and L.~Ying.
\newblock Solving for {H}igh-dimensional {C}ommittor {F}unctions {U}sing
  {A}rtificial {N}eural {N}etworks.
\newblock {\em Res. Math. Sci.}, 6:1--13, 2019.

\bibitem{Yuehaw20182}
Y.~Khoo and L.~Ying.
\newblock Switch{N}et: a {N}eural {N}etwork {M}odel for {F}orward and {I}nverse
  {S}cattering {P}roblems.
\newblock {\em SIAM J. Sci. Comput.}, 41:A3182--A3201, 2019.

\bibitem{KB2014}
D.P. Kingma and J.~Ba.
\newblock {Adam: a {M}ethod for {S}tochastic {O}ptimization}.
\newblock {\em arXiv e-prints}, page arXiv:1412.6980, 2014.

\bibitem{PhysRev2}
W.~Kohn and L.~J. Sham.
\newblock Self-consistent {E}quations {I}ncluding {E}xchange and {C}orrelation
  effects.
\newblock {\em Phys. Rev.(2)}, 140:A1133--A1138, 1965.

\bibitem{kurkova1992}
V.~K\r{u}rkov\'{a}.
\newblock Kolmogorov's {T}heorem and {M}ultilayer {N}eural {N}etworks.
\newblock {\em Neural Networks}, 5:501--506, 1992.

\bibitem{712178}
I.E. Lagaris, A.~Likas, and D.I. Fotiadis.
\newblock Artificial {N}eural {N}etworks for {S}olving {O}rdinary and {P}artial
  {D}ifferential {E}quations.
\newblock {\em IEEE Trans. Neural Networks}, 9:987--1000, 1998.

\bibitem{LuShenYangZhang2020}
J.~{Lu}, Z.~{Shen}, H.~{Yang}, and S.~{Zhang}.
\newblock {Deep {N}etwork {A}pproximation for {S}mooth {F}unctions}.
\newblock {\em arXiv e-prints}, page arXiv:2001.03040, 2020.

\bibitem{5061501}
K.S. McFall and J.R. Mahan.
\newblock Artificial {N}eural {N}etwork {M}ethod for {S}olution of {B}oundary
  {V}alue {P}roblems with {E}xact {S}atisfaction of {A}rbitrary {B}oundary
  {C}onditions.
\newblock {\em IEEE Trans. Neural Networks}, 20:1221--1233, 2009.

\bibitem{MEADE199419}
A.J. Meade Jr. and A.A. Fern\'{a}ndez.
\newblock Solution of {N}onlinear {O}rdinary {D}ifferential {E}quations by
  {F}eedforward {N}eural {N}etworks.
\newblock {\em Math. Comput. Modelling}, 20:19 -- 44, 1994.

\bibitem{Miura}
R.M. Miura.
\newblock The {K}orteweg-de{V}ries {E}quation: a {S}urvey of {R}esults.
\newblock {\em SIAM Rev.}, 18:412--459, 1976.

\bibitem{montanelli2019a}
H.~Montanelli and Q.~Du.
\newblock New {E}rror {B}ounds for {D}eep {R}e{LU} {N}etworks {U}sing {S}parse
  {G}rids.
\newblock {\em SIAM J. Math. Data Sci.}, 1:78--92, 2019.

\bibitem{Yang20192}
H.~Montanelli and H.~Yang.
\newblock Error {B}ounds for {D}eep {R}e{LU} {N}etworks {U}sing the
  {K}olmogorov--{A}rnold {S}uperposition {T}heorem.
\newblock {\em arXiv e-prints}, page arXiv:1906.11945, 2019.

\bibitem{montanelli2019b}
H.~Montanelli, H.~Yang, and Q.~Du.
\newblock Deep {R}e{LU} {N}etworks {O}vercome the {C}urse of {D}imensionality
  for {B}andlimited {F}unctions.
\newblock {\em arXiv e-prints}, page arXiv:1903.00735, 2019.

\bibitem{QuarteroniValli1994}
A.~Quarteroni and A.~Valli.
\newblock {\em Numerical Approximation of Partial Differential Equations}.
\newblock Springer, Berlin, 1994.

\bibitem{RaissiPerdikarisKarniadakis2019}
M.~Raissi, P.~Perdikaris, and G.E. Karniadakis.
\newblock Physics-informed {N}eural {N}etworks: a {D}eep {L}earning {F}ramework
  for {S}olving {F}orward and {I}nverse {P}roblems {I}nvolving {N}onlinear
  {P}artial {D}ifferential {E}quations.
\newblock {\em J. Comput. Phys.}, 378:686 -- 707, 2019.

\bibitem{RUDD2015}
K.~Rudd and S.~Ferrari.
\newblock A {C}onstrained {I}ntegration ({CINT}) {A}pproach to {S}olving
  {P}artial {D}ifferential {E}quations {U}sing {A}rtificial {N}eural
  {N}etworks.
\newblock {\em Neurocomputing}, 155:277 -- 285, 2015.

\bibitem{doi:10.1063/1.4961454}
K.~Shao, J.~Chen, Z.~Zhao, and D.~H. Zhang.
\newblock Communication: {F}itting {P}otential {E}nergy {S}urfaces with
  {F}undamental {I}nvariant {N}eural {N}etwork.
\newblock {\em J. Chem. Phys.}, 145:071101, 2016.

\bibitem{SHEKARIBEIDOKHTI2009898}
R.~Shekari~Beidokhti and A.~Malek.
\newblock {S}olving {I}nitial-boundary {V}alue {P}roblems for {S}ystems of
  {P}artial {D}ifferential {E}quations {U}sing {N}eural {N}etworks and
  {O}ptimization {T}echniques.
\newblock {\em J. Franklin Inst.}, 346:898 -- 913, 2009.

\bibitem{ShenYangZhang2019}
Z.~Shen, H.~Yang, and S.~Zhang.
\newblock {Deep {N}etwork {A}pproximation {C}haracterized by {N}umber of
  {N}eurons}.
\newblock {\em arXiv e-prints}, page arXiv:1906.05497, 2019.

\bibitem{ShenYangZhang20192}
Z.~Shen, H.~Yang, and S.~Zhang.
\newblock Nonlinear {A}pproximation via {C}ompositions.
\newblock {\em Neural Networks}, 119:74 -- 84, 2019.

\bibitem{SIRIGNANO2018}
J.~Sirignano and K.~Spiliopoulos.
\newblock {DGM}: a {D}eep {L}earning {A}lgorithm for {S}olving {P}artial
  {D}ifferential {E}quations.
\newblock {\em J. Comput. Phys.}, 375:1339 -- 1364, 2018.

\bibitem{StoerBulirsch2002}
J.~Stoer and R.~Bulirsch.
\newblock {\em Introduction to Numerical Analysis(Third Edition)}.
\newblock Springer, New York, 2002.

\bibitem{Mingui2003}
M.~Sun, X.~Yan, and R.~J. Sclabassi.
\newblock Solving {P}artial {D}ifferential {E}quations in {R}eal-time {U}sing
  {A}rtificial {N}eural {N}etwork {S}ignal {P}rocessing as an {A}lternative to
  {F}inite-element {A}nalysis.
\newblock In {\em International Conference on Neural Networks and Signal Processing, 2003.}, volume~1, pages 381--384, 2003.

\bibitem{Tang2017}
W.~Tang, T.~Shan, X.~Dang, M.~Li, F.~Yang, S.~Xu, and J.~Wu.
\newblock Study on a {P}oisson's {E}quation {S}olver {B}ased on {D}eep
  {L}earning {T}echnique.
\newblock In {\em 2017 IEEE Electrical Design of Advanced Packaging and Systems
  Symposium (EDAPS)}, pages 1--3, 2017.

\bibitem{doi:10.1063/1.5054310}
C.~Xie, X.~Zhu, D.~R. Yarkony, and H.~Guo.
\newblock Permutation {I}nvariant {P}olynomial {N}eural {N}etwork {A}pproach to
  {F}itting {P}otential {E}nergy {S}urfaces. {IV}. {C}oupled {D}iabatic
  {P}otential {E}nergy {M}atrices.
\newblock {\em J. Chem. Phys.}, 149:144107, 2018.

\bibitem{Xu1994}
J.~Xu.
\newblock A {N}ovel {T}wo-{G}rid {M}ethod for {S}emilinear {E}lliptic
  {E}quations.
\newblock {\em SIAM J. Sci. Comput.}, 15:231--237, 1994.

\bibitem{Xu1996}
J.~Xu.
\newblock Two-grid {D}iscretization {T}echniques for {L}inear and {N}onlinear
  {PDE}s.
\newblock {\em SIAM J. Numer. Anal.}, 33:1759--1777, 1996.

\bibitem{XuZhou2001}
J~Xu and A~Zhou.
\newblock A {T}wo-grid {D}iscretization {S}cheme for {E}igenvalue {P}roblems.
\newblock {\em Math. Comp.}, 70:17--25, 2001.

\bibitem{yarotsky2017}
D.~Yarotsky.
\newblock Error {B}ounds for {A}pproximations with {D}eep {R}e{LU} {N}etworks.
\newblock {\em Neural Networks}, 94:103--114, 2017.

\bibitem{yarotsky2018}
D.~Yarotsky.
\newblock Optimal {A}pproximation of {C}ontinuous {F}unctions by very {D}eep
  {R}e{LU} {N}etworks.
\newblock In {\em 31st Annual Conference on Learning Theory}, volume~75, pages
  1--11. 2018.

\bibitem{haomin}
Y.~Zang, G.~Bao, X.~Ye, and H.~Zhou.
\newblock Weak {A}dversarial {N}etworks for {H}igh-dimensional {P}artial
  {D}ifferential {E}quations.
\newblock {\em arXiv e-prints}, page arXiv:1907.08272, 2019.

\bibitem{Zhang:2018:ESP:3327345.3327356}
L.~Zhang, J.~Han, H.~Wang, W.~A. Saidi, R.~Car, and W.~E.
\newblock End-to-end {S}ymmetry {P}reserving {I}nter-atomic {P}otential
  {E}nergy {M}odel for {F}inite and {E}xtended {S}ystems.
\newblock In {\em  Advances in Neural Information Processing Systems 31}, pages 4436--4446, 2018. 

\bibitem{Zhou2004}
A.~Zhou.
\newblock An {A}nalysis of {F}inite-dimensional {A}pproximations for the
  {G}round {S}tate {S}olution of {B}ose-{E}instein {C}ondensates.
\newblock {\em Nonlinearity}, 17:541--550, 2004.

\bibitem{ZhuLin1989}
C.~Zhu and Q.~Lin.
\newblock {\em The {H}yperconvergence {T}heory of {F}inite {E}lements}.
\newblock Hunan Science and Technology Publishing House, Changsha, 1989.

\bibitem{ZienkiewiczTaylorZhu2005}
O.~C. Zienkiewicz, R.~L. Taylor, and J.Z. Zhu.
\newblock {\em The Finite Element Method (Sixth Edition)}.
\newblock Butterworth-Heinemann, Oxford, 2005.

\end{thebibliography}

\newpage
\appendix{{\bf\large Appendix}}

\section{Proof of Theorem \ref{thm: semilinear PDE}}

Making use of the assumption A2 in Section \ref{sec:sl}, we immediately have the following result (see \cite{Xu1994}).
\begin{lemma}\label{lem: infsup}
	Let $u \in W^{2,\infty}(\Omega)$ be a solution of \eqref{eq: semilinear}. There exist
	two positive constant $C_1^u$ and $C_2^u$ such that if a function $v$ satisfies that $\| v - u \|_{0,\infty} \leq C_1^u$, then
	\[
		C_2^u \| \phi \|_1 \leq \sup_{\chi\in V} \frac{a_v(\phi,\chi)}{\| \chi \|_1},
		\quad \phi \in V.
	\]
\end{lemma}

The mathematical analysis for the finite element method \eqref{fem-semilinear} is very technical and has been established in \cite{Xu1996}. In the following two lemmas, we collect some results which will be used frequently later on.
\begin{lemma}\label{lem: discrete-infsup}
	Let $u \in W^{2,\infty}(\Omega)$ be a solution of \eqref{v-form}. Then there exists a constant $h_0>0$ and a positive constant $C_3^u$ independent of $h$ such that if $h<h_0$ and a function $v$ satisfies $\| v - u \|_{0,\infty} \leq C_1^u$, then
	\[
		C_3^u \| \phi \|_1 \leq \sup_{\chi\in V_h} \frac{a_v(\phi,\chi)}{\| \chi \|_1},
		\quad \phi \in V_h.
	\]
\end{lemma}
This lemma can be derived by using Lemma \ref{lem: infsup} and the arguments for proving Lemma 2.2 in \cite{Xu1996}.

\begin{lemma}
\label{convergence}
	Let $u \in W^{2,\infty}(\Omega)$ be a solution of \eqref{v-form}. Then there exists a positive constant $h_1<h_0$ such that if $h<h_1$, the finite element method \eqref{fem-semilinear} has exactly one solution $u^h$ satisfying
	\[
		\| u - u^h \|_1 \leq C_4^u
	\]
for a positive constant $C_4^u$ independent of $h$. Moreover,
	\[
		\lim_{h \to 0^+}\| u -u^h \|_1 = 0.
	\]
\end{lemma}

\begin{lemma}\label{lem: error estimates}
	Let $u \in W^{2,\infty}(\Omega)$ be a solution of \eqref{v-form} and $u^h$ be the finite element method of \eqref{fem-semilinear}. Then, there exists a positive constant { $h_2<h_1$ such that if $h<h_2$},
	\begin{align*}
		& \| u - u^h \|_{0,\infty} \lesssim h^2 \quad \quad \quad \quad \,\, when \ d=1,
		\\
		& \| u - u^h \|_{0,\infty} \lesssim h^2|\log(h)| \quad when \ d=2.
	\end{align*}
\end{lemma}
\begin{proof}
The estimate for $d=2$ is given in \cite{Xu1996}. We will derive the estimate for $d=1$ following some ideas in \cite{Xu1996}. To this end, we first introduce an elliptic projection operator $P_h$ such that if $u\in V$, then $P_h u \in V_h$ satisfies
\begin{equation}
\label{e-projection}
	(\nabla (P_h u ), \nabla \chi) = (\nabla u, \nabla \chi), \quad \chi \in V_h.
\end{equation}
It is shown in \cite{ZhuLin1989} that
\begin{equation}
\label{infty-norm}
	\| u - P_h u \|_{0,\infty} \lesssim \| u - P_h u \|_1 \lesssim h^2 |u|_{2,\infty}.
\end{equation}

On the other hand, the finite element solution $u^h$ satisfies
\[
	(\nabla u^h, \nabla \chi) + (f(u^h), \chi) = 0, \quad \chi\in V_h.
\]
Recalling the relation \eqref{e-projection} and using \eqref{v-form} gives
\[
	(\nabla (P_h u ), \nabla \chi) = (\nabla u, \nabla \chi) = - (f(u), \chi), \quad \chi\in V_h.
\]
Hence, subtracting the last two equations and using Taylor's expansion yield
\[
	a_{P_h u} ( (u^h - P_h u), \chi) = -\frac{1}{2}(f^{\prime\prime}(\zeta)(u^h - P_h u)^2, \chi) - (f(P_h u) - f(u), \chi), \quad \chi \in V_h,
\]
where $\zeta = \nu(x) u^h(x) + (1-\nu(x)) (P_h u)(x)$ for some $\nu(x) \in (0,1)$. We have by the estimate \eqref{infty-norm} that there exists a positive constant { $\bar{h}<h_1$} such that if $h<\bar{h}$, $\|P_hu-u\|_{0,\infty}\le C_1^u$. In this case, it follows from Lemma \ref{lem: discrete-infsup} and the last equation that
\begin{align}
	\| u^h - P_h u \|_1
	& \lesssim \sup_{\chi \in V_h} \frac{a_{P_h u} ( (u^h - P_h u), \chi)}{\| \chi\|_1} \nonumber
	\\
	& \lesssim \| u^h - P_h u \|^2_{0,\infty} + \| f(u) - f(P_h u) \|_{0,\infty} \nonumber
	\\
	& \lesssim \| u^h - P_h u \|^2_1 + \| u - P_h u \|_{0,\infty} \nonumber
	\\
	& \le C \| u^h - P_h u \|_1^2 + C \| u - P_h u \|_{0,\infty}, \label{estimate1}
\end{align}
where $C>0$ is a generic constant. On the other hand, it follows from Lemma \ref{convergence} and the estimate \eqref{infty-norm} that
\[
\|u^h - P_h u\|_1\le \|u-u^h\|_1+\|u-P_hu\|_1\to 0\quad \mbox { as } h\to 0^+.
\]
Hence, there exists a positive constant { $h_2<\bar{h}$ such that if $h<h_2$}, then $C\|u^h - P_h u\|_1<1/2$, which combined with \eqref{estimate1} readily implies
\[
	\| u^h - P_h u \|_1  \lesssim \| u - P_h u \|_{0,\infty}.
\]
Therefore, by the Sobolev embedding theorem and the estimate \eqref{infty-norm},
\begin{align*}
	\| u - u^h \|_{0,\infty}
	& \le \| u -P_h u \|_{0,\infty} + \| P_h u - u^h \|_{0,\infty}
	\\
	& \lesssim \| u -P_h u \|_{0,\infty} + \| P_h u - u^h \|_1\lesssim \| u -P_h u \|_{0,\infty}\lesssim h^2 | u |_{2,\infty},
	\end{align*}
as required.
\end{proof}	
\begin{remark}
As shown in \cite{BrennerScott2008}, we require to make certain strong regularity assumption on the solution to problem \eqref{v-form} for $d=3$, so as to derive the maximum norm estimate for the related finite element method. To avoid too technical treatment, we skip the further discussion in this case.
\end{remark}

The following result plays an important role in the convergence analysis of Int-Deep for solving the finite element method \eqref{fem-semilinear}, which will be introduced later on.

\begin{lemma}\label{lem: recursion}
Let $u \in W^{2,\infty}(\Omega)$ be a solution of \eqref{v-form}. Let $u^h$ be an approximation of $u$ obtained by the finite element method \eqref{fem-semilinear}. Assume that { $h<h_2$ with $h_2$} the same as given in Lemma \ref{lem: error estimates}. For any $v\in V_h \cap B(u)$, define a mapping $Tv = v+w$, where $w\in V_h$ is uniquely determined by
\begin{equation}
\label{linearized-problem}
	(\nabla w, \nabla \chi) + (f'(v)w, \chi) = -(\nabla v, \nabla \chi) - (f(v), \chi), \quad \chi \in V_h.
\end{equation}
Then there holds
\begin{align*}
	\|u^h - Tv \|_1 \lesssim \| u^h - v \|_{0,p}^2 \lesssim \| u^h - v \|_1^2.
\end{align*}
\end{lemma}

\begin{proof}
Recalling the definition \eqref{fem-semilinear}, we know
\[
	(\nabla u^h, \nabla \chi) + (f(u^h), \chi) = 0, \quad \chi \in V_h,
\]
Subtracting the above equation from \eqref{linearized-problem} and reorganizing terms, we find
\[
	(\nabla E_h, \nabla \chi) + (f(u^h)-f(v)-f'(v)w, \chi) = 0, \quad \chi\in V_h,
\]
where $E_h = u^h - Tv$. Furthermore, use Taylor's expansion to get
\[
	a_v(E_h, \chi)=-\frac{1}{2}(f^{\prime\prime}(\zeta)(u^h-v)^2, \chi), \quad \chi\in V_h,
\]
where $\zeta = (1-\nu(x))v(x) + \nu(x) u^h(x)$ with $\nu(x) \in (0,1)$. Since { $h<h_2$} and $v\in B(u)$, from Lemma \ref{convergence} it follows that { $\|\zeta\|_{0,\infty}$} is uniformly bounded with respect to $h$. Hence, by the H$\ddot{\text{o}}$lder inequality and the Sobolev embedding theorem, for $p>2$ and any $\chi\in V_h$,
\begin{align*}
	a_v(E_h, \chi)&  \lesssim \int_{\Omega} (u^h - v)^2 |\chi| \dx
	\\
	& \lesssim \| (u^h - v)^2 \|_{0,p/2} \| \chi \|_{0,p/p-2} \lesssim \| u^h - v \|_{0,p}^2 \| \chi \|_1,
\end{align*}
where the generic constant is independent of $h$ but depends on $p$. The combination of the last estimate with Lemma~\ref{lem: discrete-infsup} immediately implies
\[
	\| E_h \|_1 \lesssim \| u^h - v \|_{0,p}^2 \lesssim \| u^h - v \|_1^2,
\]
as required.
\end{proof}

Now we are ready to prove Theorem \ref{thm: semilinear PDE}.
\begin{proof}
Denote $E_k^h = u^h - u_k^h$. Applying Lemma~\ref{lem: recursion} gives rise to
\begin{equation}\label{eq: semilinear recursion}
	\| E_{k+1}^h \|_1 \lesssim \| E_k^h \|_{0,p}^2 \lesssim \| E_k^h \|_1^2.
\end{equation}
	
When $d=1$, By the Sobolev embedding theorem and \eqref{eq: semilinear recursion},
\[
	\| E_{k+1}^h \|_{0,\infty}
	\lesssim \| E_{k+1}^h \|_1
	\lesssim \| E_k^h \|_{0,p}^2
	\lesssim \| E_k^h \|_{0,\infty}^2.
\]
This means there exists a positive constant $c_1$ such that $\| E_{k}^h \|_{0,\infty} \leq c_1 \| E_{k-1}^h \|_{0,\infty}^2$. Hence,
\[
	c_1\| E_{k}^h \|_{0,\infty}
	\leq ( c_1 \| E_{k-1}^h \|_{0,\infty})^2
	\leq ( c_1 \| E_0^h \|_{0,\infty})^{2^{k}}.
\]
On the other hand, it follows from Lemma~\ref{lem: error estimates} and error estimates for the interpolation operator $I_h$ (cf. \cite{BrennerScott2008,Ciarlet1978}) that
\begin{align*}
	\| E_0^h \|_{0,\infty}
	& = \| u^h - u_0^h \|_{0,\infty}
	\\
	& \leq \| u^h - u \|_{0,\infty} + \| u - I_hu \|_{0,\infty} + \| I_hu - I_h u^{DL} \|_{0,\infty}
	\\
	& \leq c_0(h^2 + \delta),
\end{align*}
where $c_0>0$ is a generic constant. Let $\beta_1 = c_1c_0(h^2 + \delta)$. Then
\begin{equation}
\label{error-1}
\| E_{k}^h \|_{0,\infty} \leq \beta_1^{2^k}/c_1.
\end{equation}

When $d=2$, for any given number $p>2$, we have by the Sobolev embedding theorem and \eqref{eq: semilinear recursion} that
\[
	\| E_{k+1}^h \|_{0,p} \lesssim \| E_{k+1}^h \|_1^2 \lesssim \| E_k^h \|_{0,p}^2,
\]
which implies $c_2\| E_{k}^h \|_{0,p} \leq (c_2 \| E_{0}^h \|_{0,p})^{2^k}$ for a generic positive constant $c_2$. In addition, by Lemma~\ref{lem: error estimates} and error estimates for $I_h$,
\begin{align*}
	\| E_0^h \|_{0,p}
	& = \| u^h - u_0^h \|_{0,p}
	\\
	& \leq \| u^h - u \|_{0,p} + \| u - I_hu \|_{0,p} + \| I_hu - u_0^h \|_{0,p}
	\\
	& \lesssim \| u^h - u \|_{0,\infty} + \| u - I_hu \|_{0,\infty} + \| I_hu - u_0^h \|_{0,\infty}
	\\
	& \leq c_3(h^2|\log h| + \delta),
\end{align*}
where $c_3>0$ is a generic constant independent of $h$ and $\delta$ but depending on  $p$. Let $\beta_2 = c_2c_3(h^2|\log h| + \delta)$. Then
\[
\| E_{k}^h \|_{0,p} \leq \beta_2^{2^k}/c_2,
\]
and further by the inverse inequality for finite elements,
\begin{equation}
\label{error-2}
	\| E_k^h \|_{0,\infty}
	\lesssim h^{-2/p}\| E_k^h \|_{0,p}
	\lesssim h^{-2/p}\beta_2^{2^k}.
\end{equation}

Now, we have by \eqref{error-1}, \eqref{error-2} and Lemma \ref{lem: error estimates} that
\[
	\| u - u_k^h \|_{0,\infty}
	\le \| u - u^h \|_{0,\infty} + \| u^h - u_k^h \|_{0,\infty}
	\lesssim h^2 + \beta_1^{2^k} \quad \mbox{for } d=1,
\]
and
\[
	\| u - u_k^h \|_{0,\infty}
	\le \| u - u^h \|_{0,\infty} + \| u^h - u_k^h \|_{0,\infty}
	\lesssim h^2 |\log h| + h^{-2/p}\beta_2^{2^k} \quad \mbox{for } d=2.
\]
The proof is complete.

\end{proof}

\section{Proof of Theorem \ref{thm: eigenvalue problem}}

The proof of Theorem \ref{thm: eigenvalue problem} is rather involved and requires the following elementary but nontrivial result as a key bridge to produce the optimal convergence analysis.
{ 
\begin{lemma}\label{B sequence}
Let $\{ a_k \}$ be a sequence satisfying $a_{k+1} \leq a_k^2 +b$ for $k=0,1,2,\cdots$. If $0\le a_0 < 1/2$ and $0 < b < 1/4$, then
\begin{equation}
\label{upper-bound}
  a_k \leq a_0^{2^k} + \left(2+ \frac{1}{1-2a_0}\right)b.
\end{equation}
\end{lemma}
\begin{proof}
First of all, construct an auxiliary sequence $\{ A_k \}$ generated by
\begin{equation}\label{Ak def}
  A_{k+1} = A_k^2 +b, \quad k=0,1,2,\cdots;\quad A_0=a_0.
\end{equation}
It is easy to check $a_k \leq A_k$ for all nonnegative integer $k$. The recursive relation \eqref{Ak def} is a fixed point iteration corresponding to the following fixed point equation
\begin{equation}
\label{quadratic}
 A=A^2+b\Rightarrow A^2 - A + b =0,
\end{equation}
which has exactly two fixed points
\begin{align*}
  \alpha_0 = \frac{2b}{1+\sqrt{1-4b}}
  \quad \mbox{and} \quad
  \alpha_1 = \frac{1+\sqrt{1-4b}}{2}.
\end{align*}

Next, let us study the boundedness of the sequence $\{A_k\}$. When $A_0 \leq \alpha_0$, we have by mathematical induction that $A_k \leq \alpha_0$ for $k=0,1,2,\cdots$. In fact, if $k=0$  the statement is true. Now we assume it holds for $k=m$ with $m$ an any given natural number, i.e. $A_m \leq \alpha_0$.  Then, observing that $\alpha_0$ satisfies the equation \eqref{quadratic}, we immediately know
\[
  A_{m+1} = A_m^2 + b \leq \alpha_0^2 +b = \alpha_0.
\]
So the statement is really true by the principle of mathematical induction. On the other hand, it is easy to check by a direct computation that $\alpha_0\leq 2b$. Hence, if $A_0 \leq \alpha_0$, there holds
\begin{equation}\label{case1}
a_k \leq A_k \leq \alpha_0 \leq 2b.
\end{equation}

When $\alpha_0 \leq A_0 \leq \alpha_1$, we use mathematical induction again to know $\alpha_0 \leq A_k \leq \alpha_1$ for any nonnegative integer $k$. Let us now consider a function $f(x)$ defined by
\[
  f(x)= x + \frac{b}{x}, \quad x \in [\alpha_0, \alpha_1].
\]
Clearly, it is convex over $[\alpha_0, \alpha_1]$ and hence must take the maximum value over the interval at one of the two end points. However, recalling the definitions of $\alpha_0$ and $\alpha_1$, we easily know
\[
f(\alpha_0)=f(\alpha_1)=1,
\]
which implies $f(x)\leq 1$ for $x\in [\alpha_0, \alpha_1]$. Thus, since $A_k\in  [\alpha_0, \alpha_1]$, we find
\[
  \frac{A_{k+1}}{A_k} = \frac{A_k^2+b}{A_k} =f(A_k)\leq 1;
\]
in other words, $\{A_k\}$ is a decreasing sequence. Therefore, the result $A_{k} \leq a_0$ also holds in this case. To sum up, we find that if $a_0<1/2<\alpha_1$, $A_{k} \leq a_0$ for all nonnegative integer $k$.

To further our analysis, we require to construct another auxiliary sequence generated by
\begin{equation}
\label{Bk def}
B_{k+1} = B_k^2,\quad k=0,1,2,\cdots;\quad B_0 = a_0.
\end{equation}
It is evident that $B_k=B_0^{2^k}=a_0^{2^k}$ and $B_k \leq A_k$. Denote $d_k=A_k - B_k\ge 0$. Recalling the definitions \eqref{Ak def} and \eqref{Bk def}, and using the fact that $A_k\leq a_0$ obtained above, we find
\[
  d_{k+1} = A_k^2 + b - B_k^2 = d_k(A_k+B_k)+b \leq 2a_0d_k+b,
\]
which readily yields
\[
  d_k \leq (2a_0)^{k}d_0 + \frac{1-(2a_0)^k}{1-2a_0}b\le \frac 1 {1-2a_0}b
\]
by noting that $d_0 = 0$. Therefore,
\begin{equation*}
  A_k = B_k + d_k \leq B_0^{2^k} + \frac{1-(2a_0)^k}{1-2a_0}b = a_0^{2^k} + \frac{1-(2a_0)^k}{1-2a_0}b
\end{equation*}
and
\begin{equation}\label{case2}
  a_k \leq A_k \leq a_0^{2^k} + \frac{1}{1-2a_0}b.
\end{equation}
We then obtain the required estimate by combining \eqref{case1} and \eqref{case2}.
\end{proof}
}

Now we are ready to present the proof of Theorem \ref{thm: eigenvalue problem}.

\begin{proof}
For any $\phi \in V_h$, we know
\begin{align*}
 a(P_h u - u_{k+1}^h, \chi)
 &= a(u,\chi) - a(u_{k+1}^h ,\chi)
 = \lambda(u, \chi) - \lambda_k^h(u_k^h, \chi)
 \\
 &= (\lambda - \lambda_k^h)(u,\chi) + \lambda_k^h(u - u_k^h, \chi).
\end{align*}
Choosing $\chi = P_h u - u_{k+1}^h$, we have by the coerciveness of $a(\cdot, \cdot)$ that
\begin{align*}
 \alpha_0 \| P_h u - u_{k+1}^h \|_1^2
 & \leq a(P_h u - u_{k+1}^h, P_h u - u_{k+1}^h)
 \\
 &\leq |\lambda - \lambda_k^h|(u,P_h u - u_{k+1}^h) + |\lambda_k^h|(u - u_k^h, P_h u - u_{k+1}^h)
 \\
 & \lesssim \left( |\lambda - \lambda_k^h| \| u \|_{-1} + |\lambda_k^h - \lambda + \lambda| \| u - u_{k}^h \|_{-1} \right) \|P_h u - u_{k+1}^h\|_1,
\end{align*}
which combined with the assumption { $\|u_k^h\|_0\le \varepsilon_1$} implies
\begin{align*}
 \| P_h u - u_{k+1}^h \|_1
 & \lesssim |\lambda - \lambda_k^h| \| u \|_{0} + \left( | \lambda - \lambda_k^h | + | \lambda| \right) \| u - u_{k}^h \|_{0}
 \\
 & \lesssim |\lambda - \lambda_k^h| + | \lambda - \lambda_k^h | \left( \| u\|_0+\|u_k^h\|_0\right) + \| u - u_{k}^h \|_0
 \\
 & \lesssim |\lambda - \lambda_k^h| + \| u - u_{k}^h \|_0.
\end{align*}
Hence, by the error estimate for $P_h$ and the triangle inequality,
\begin{equation}\label{pf: eigenvector recursion}
  \| u - u_{k+1}^h \|_1 \leq \| u -P_h u \|_1 + \| P_h u - u_{k+1}^h \|_1 \lesssim h + |\lambda - \lambda_k^h| + \| u - u_{k}^h \|_0.
\end{equation}
Using the duality argument to the equation determining $u_{k+1}^h$, we deduce from \eqref{pf: eigenvector recursion} and the Cauchy inequality that
\begin{align}
\label{pf: eigenvector}
  \| u - u_{k+1}^h \|_0
 & \lesssim h \| u - u_{k+1}^h \|_1\lesssim h^2 + h
 \left( |\lambda - \lambda_k^h| + \| u - u_{k}^h \|_0 \right)  \nonumber
 \\
 & \lesssim h^2+\left( |\lambda - \lambda_k^h| + \| u - u_{k}^h \|_0 \right)^2.
\end{align}
By Lemma~\ref{lem: eigenvlue} and assumption $\| u_{k+1}^h \|_0\ge \varepsilon_0 > 0$,
\[
 |\lambda-\lambda_{k+1}^h| = \left | |u_{k+1}^h - u|_1^2 - \lambda \| u_{k+1}^h - u \|_0^2 \right|/\|u_{k+1}^h\|_0^2
  \lesssim \| u_{k+1}^h - u \|_1^2.
\]
Inserting this into \eqref{pf: eigenvector recursion} gives
\[
 | \lambda - \lambda_{k+1}^h |
 \lesssim \| u - u_{k+1}^h \|_1^2
 \lesssim \left( h + |\lambda - \lambda_k^h| + \| u - u_{k}^h \|_0 \right)^2
 \lesssim h^2 + \left(|\lambda - \lambda_k^h| + \| u - u_{k}^h \|_0 \right)^2,
\]
which combined with \eqref{pf: eigenvector} implies
\begin{equation}\label{pf: eigenvalue eigenvector recursion}
 | \lambda - \lambda_{k+1}^h | + \| u - u_{k+1}^h \|_0 \lesssim h^2 + \left(|\lambda - \lambda_k^h| + \| u - u_{k}^h \|_0 \right)^2.
\end{equation}
Let $e_k = |\lambda - \lambda_k^h| + \| u - u_{k}^h \|_0$. Then the estimate \eqref{pf: eigenvalue eigenvector recursion} can be expressed as
\[
 e_{k+1} \leq \bar{c}_4h^2 + c_4e_k^2,
\]
where $c_4$ and $\bar{c}_4$  are two positive generic constants. The above inequality can expressed as
\begin{equation}
\label{eigen-recursive}
 c_4 e_{k+1} \leq c_5 h^2 + (c_4e_k)^2,
\end{equation}
where  $c_5=c_4\bar{c}_4$.

On the other hand, by assumption and using error estimates of $I_h$, we know there exists a constant $\tilde{h}_1>0$ such that if $h<\tilde{h}_1$, there holds
\[
 e_0
 = |\lambda - \lambda_0^h| + \| u - u_{0}^h \|_0
 \leq |\lambda - \lambda_0^h| + \| u - I_hu\|_0  + \| I_h u - u_{0}^h \|_0
 \le c_6(\tilde{\delta}+ h^2),
\]
where $c_6>0$ is a generic constant. For clarity, we want to point out all the generic constants given above are independent of the finite element mesh size $h$.

{  Next, we choose two positive constants $\tilde{h}_0$ and $\tilde{\delta}_0$ such that
\begin{equation}
\label{constraints}
\tilde{h}_0<\tilde{h}_1,\quad c_4c_6(\tilde{\delta_0}+\tilde{h_0}^2)<1/2\quad \mbox{and} \quad  c_5\tilde{h_0}^2 < 1/4.
\end{equation}

For the recursive estimate \eqref{eigen-recursive}, we set $a_k=c_4e_k$ and $b=c_5 h^2$. Then, if $h<\tilde{h}_0$ and $\tilde{\delta}<\tilde{\delta}_0$, the conditions \eqref{constraints} hold, or equivalently the assumptions in Lemma \ref{B sequence} hold, so it follows from \eqref{upper-bound} that
\[
 c_4e_k \leq (c_4e_0)^{2^k} + c_5\left(2+ \frac{1}{1-2c_4e_0}\right)h^2 = (c_4e_0)^{2^k} + c_5\left(3+ \frac{2c_4e_0}{1-2c_4e_0}\right)h^2,
\]
which immediately yields
\[
 |\lambda - \lambda_k^h| + \| u - u_{k}^h \|_0 \lesssim \beta_3^{2^k} + h^2.
\]
The proof is complete.}
\end{proof}

\end{document}